\documentclass[12pt,twoside]{article}

\usepackage[T1]{fontenc}
\usepackage[latin1]{inputenc}
\usepackage[english]{babel}
\usepackage[babel]{csquotes}

\usepackage{cite}

\usepackage{amssymb}
\usepackage{amsmath}
\usepackage{amsthm}
\usepackage{latexsym}
\usepackage{graphicx}
\usepackage{mathrsfs}
\usepackage{mathrsfs}
\usepackage{hyperref}
\usepackage{pgfplots}

\usepackage{color}

\definecolor{viola}{rgb}{0.3,0,0.7}
\definecolor{lilla}{rgb}{0.8,0,0.8}
\definecolor{ciclamino}{rgb}{0.5,0,0.5}
\definecolor{blu}{rgb}{0,0,0.7}
\definecolor{verde}{rgb}{0,0.5,0.2}
\definecolor{rosso}{rgb}{0.8,0,0}
\definecolor{gius}{rgb}{0.0,0.5,0.5}
\definecolor{gray}{rgb}{0.9,0.9,0.9}

\def\luca #1{{\color{blu}#1}}
\def\pier #1{{\color{rosso}#1}}

\def\gius#1{{\color{gius}#1}}
\def\luca #1{#1}
\def\pier #1{#1}

\def\gius #1{#1}

\theoremstyle{plain}
\newtheorem{thm}{Theorem}[section]
\newtheorem{lem}[thm]{Lemma}
\newtheorem{prop}[thm]{Proposition}

\theoremstyle{definition}
\newtheorem{rmk}[thm]{Remark}

\def\enne{\mathbb{N}}
\def\zeta{\mathbb{Z}}

\def\erre{\mathbb{R}}

\def\eps{\varepsilon}

\def\beq{\begin{equation}}
\def\eeq{\end{equation}}

\def\to{\rightarrow}
\def\wto{\rightharpoonup}
\def\wstarto{\stackrel{*}{\rightharpoonup}}
\def\embed{\hookrightarrow}
\def\cembed{\stackrel{c}{\hookrightarrow}}

\def\norm #1{\left\|#1\right\|}

\def\ip #1#2{\left<#1,#2\right>}

\newcommand\ud{u_\delta}
\newcommand\xid{\xi_\delta}
\newcommand\mud{\mu_\delta}

\begin{document}
\begin{center}

{\huge\rm Bounded solutions and their asymptotics\\[0.2cm]
for a doubly nonlinear Cahn\pier{--}Hilliard system}
\\[1cm]
{\large\sc Elena Bonetti}\\
{\normalsize e-mail: {\tt elena.bonetti@unimi.it}}\\[.25cm]
{\large\sc Pierluigi Colli}\\
{\normalsize e-mail: {\tt pierluigi.colli@unipv.it}}\\[.25cm]
{\large\sc Luca Scarpa}\\
{\normalsize e-mail: {\tt luca.scarpa@univie.ac.at}}\\[.25cm]
{\large\sc Giuseppe Tomassetti}\\
{\normalsize e-mail: {\tt giuseppe.tomassetti@uniroma3.it}}\\[.25cm]
$^{(1)}$
{\small Dipartimento di Matematica ``F.Enriques'', Universit\`a degli 
Studi di Milano}\\ 
{\small Via Saldini 50, 20133 Milano, Italy}\\[.2cm]
$^{(2)}$
{\small Dipartimento di Matematica ``F. Casorati'', Universit\`a di Pavia}\\
{\small \pier{Via Ferrata 5}, 27100 Pavia, Italy}\\[.2cm]
$^{(3)}$
{\small Faculty of Mathematics, University of Vienna}\\
{\small Oskar-Morgenstern-Platz 1, 1090 Vienna, Austria}\\[.2cm]
$^{(4)}$
{\small Dipartimento di Ingegneria - Sezione Ingegneria Civile}\\
{\small Universit\`a degli Studi ``Roma Tre'', Via Vito Volterra 62, Roma, Italy}

\end{center}

\begin{abstract}
In this paper we deal with a doubly nonlinear Cahn\pier{--}Hilliard system, 
where both an internal constraint 
on the time derivative of the concentration 
and a potential for the concentration are introduced. 
The definition of the chemical potential includes two regularizations: a viscosity and a diffusive term. 
\gius{First of all, 
we prove existence and uniqueness of a bounded solution to the
system using a 
nonstandard maximum-principle argument for 
time-discretizations of doubly nonlinear equations.
Possibly including singular potentials, this novel 
result brings improvements over previous 
approaches to this problem.}
Secondly, under suitable assumptions on the data, 
we show the convergence of solutions to the respective limit problems
once either of the two regularization parameters vanishes.
\end{abstract}


\pagestyle{myheadings}
\newcommand\testopari{\sc Asymptotics for a doubly nonlinear CH system}
\newcommand\testodispari{\sc Bonetti -- Colli -- Scarpa -- Tomassetti}
\markboth{\testodispari}{\testopari}


\thispagestyle{empty}

\section{Introduction}\label{intro}
\setcounter{equation}{0}

The main focus of this paper is the asymptotic behaviour, 
when either of the positive parameters $\varepsilon$ or $\delta$ converges to zero, of the following system:
\begin{align}
\label{eq1:intro}
\partial_t u - \Delta\mu = 0 \qquad&\text{in } \Omega\times(0,T)\,,\\
\label{eq2:intro}
\mu \in \eps\partial_t u + \beta(\partial_t u) - \delta\Delta u +\psi'(u)+ g \qquad&\text{in } \Omega\times(0,T)\,,\\
\label{boundary:intro}
\partial_{\bf n} u =0\,, \quad \mu=0 \qquad&\text{in } \partial\Omega\times(0,T)\,,\\
\label{init:intro}
u(0)=u_0 \qquad&\text{in } \Omega\,,
\end{align}  
where $\Omega\subset\erre^3$ is a smooth bounded domain and $T>0$
is a fixed final time.
Here $\beta$ is a maximal monotone graph, $\psi'$ is the derivative of a possibly non-convex potential, and $g$ is a forcing term. We shall address the unknowns $u$ and $\mu$ as, respectively, the \emph{concentration} and the \emph{chemical potential}.

System \eqref{eq1:intro}--\eqref{eq2:intro} is a modification of the celebrated Cahn\pier{--}Hilliard (C-H) system, a phenomenological model that has its origin in the work of J.W.~Cahn~\cite{Cahn1961} concerning the effects of interfacial energy on the stability of spinodal states in solid binary solutions. 
Cahn's work built upon previous collaboration with J.W.~Hilliard \cite{Cahn1958}, where the functional
\begin{equation}\label{eq:free_energy}
\mathcal F(u)=\int_\Omega \left(\psi(u)+\frac\delta2 |\nabla u|^2\right)
\end{equation}
was proposed as a model for the (Helmholtz) free energy of a non-uniform system 
whose composition is described by the scalar field $u$. In this functional, the \emph{bulk energy} 
$\psi(u)$ represents the specific energy of a uniform solution, typically a non-convex function. 
The quadratic \emph{gradient energy} $\frac \delta 2 |\nabla u|^2$ takes into account 
microscopic mechanisms that penalize spatial variation of composition, 
and that are responsible for the presence of interfacial energy between phases at the macroscopic scale. 
Cahn showed that certain states, which would be unstable if only the 
bulk energy was accounted for, are in fact stable under local perturbations, 
when the gradient energy is included in the picture.

Besides being a fundamental contribution to Materials Science, the C-H system has had considerable success in many other branches of Science and Engineering where segregation of a diffusant leads to pattern formation, such as population dynamics\cite{Liu2016}, image processing\cite{Bertozzi2007}, \pier{dynamics for mixtures of fluids\cite{GGM},}
tumor modelling\pier{\cite{Agosti2017, CGRS1, CGRS2}}, to name a few.

In the derivation of the Cahn\pier{--}Hilliard system, the variation of the free energy \eqref{eq:free_energy}, namely,
$$\mu_{\textsc{c-h}}:=\frac{\delta\mathcal F}{\delta u}(u)=\psi'(u)-\delta\Delta u,$$ 
is the chemical potential that drives the space-time evolution of the concentration $u$ through the diffusion equation \eqref{eq1:intro}. Here we have written it 
after rescaling time, so that the mobility 
(which we assume to be constant) is numerically equal to the unity 
(equivalently, one may look at the Cahn\pier{--}Hilliard system  as the gradient flow, 
with respect to the norm of the dual of a Sobolev space \cite{Fife2001}). 
The connection between \eqref{eq1:intro}--\eqref{eq2:intro} and the C-H system 
is more transparent if we rewrite \eqref{eq2:intro} as a pair of an equation and an inclusion:
\[
\mu=\mu_{\textsc{c-h}}+\varepsilon\partial_t u+\xi,\qquad \xi\in\beta(\partial_t u).
\]
The additional terms on the right-hand side do not affect the energy, but rather the dissipation. This 
is evident from the energetic estimate 
\begin{equation}\label{eq:energy_estimate}
\frac d {dt}{\mathcal F}(u(t))+\int_\Omega |\nabla\mu|^2+(\varepsilon\partial_t u+\beta(\partial_t u))\partial_t u\le-\int_\Omega g \partial_t u,
\end{equation}
which is obtained by testing the first equation by $\mu$, the second equation by $-\partial_t u$, and by adding the resulting equations.

Since the original work of Cahn, innumerable
generalizations of the C-H system have been proposed in the literature. \pier{They are
so} many that it would be difficult to provide a comprehensive account in the present context. We prefer to refer to the review \cite{Miranville2017}. In this respect it is worth mentioning that a systematic procedure to derive and generalize the C-H system has been proposed by M.E.~Gurtin \cite{Gurtin1996}, by extending the thermodynamical framework of continuum mechanics, as also reported in \cite{Miranville2000}. 	\pier{Let us also mention an alternative approach due to Podio-Guidugli~\cite{Podio} leading to another
viscous C-H system of nonstandard type~\cite{CGPS1, CGPS2}.}

In this sea of literature, the problem that we consider belongs to 
the class of doubly-nonlinear Cahn\pier{--}Hilliard systems, characterized 
by nonlinearity both on the instantaneous value $u$ of the concentration 
and on its time derivative $\partial_t u$. 
The particular form \eqref{eq1:intro}--\eqref{eq2:intro} has been the object 
of mathematical investigation in \cite{Miranville2010} 
with Neumann homogeneous conditions for the chemical potential, 
and in a previous paper of ours \cite{bcst1}, 
where a discussion of its thermodynamical consistency can also be found. 
\luca{The system \eqref{eq1:intro}--\eqref{eq2:intro} has also been
studied in \cite{scar19} under dynamic boundary conditions.}
A similar system was investigated in \cite{Miranville2009}, 
where the nonlinearity $\beta(\partial_t u)$ is replaced by $\partial_t\alpha(u)$.
Among other mathematical work on the C-H system related to the present paper, 
we mention the contributions
by Novick-Cohen and al.  \cite{Novic1988viscous,NovicP1991TAMS} 
on the viscous C-H equation, which is obtained in the case $\beta=0$
removing the nonlinear viscosity contribution. 

\luca{In all of the above-mentioned results, existence of solutions for the 
system \eqref{eq1:intro}--\eqref{eq2:intro} is proved under 
some polynomial growth assumptions
either on the nonlinearity $\beta$ acting
on the viscosity or on the nonlinearity $\psi$.
While this is certainly satisfactory in providing 
some first existence results, 
on the other hand it would be desirable to obtain 
well-posedness for the system even for possibly singular 
choices of the nonlinearities. Indeed, 
this is not only interesting from the mathematical perspective, 
but especially in the direction of applications: it
is well-known in fact that the most
physically-relevant choice for the double-well potential $\psi$
is the so-called logarithmic one, defined as}
\begin{align*}
  &\luca{\psi_{log}(r):=\frac{c}{2}\left[(1+r)\ln(1+r) + (1-r)\ln(1-r)\right] - \frac{c_0}{2}r^2\,, \quad r\in(-1,1)\,,}\\ 
  &\qquad \pier{\hbox{with }\, 0<c<c_0\,.}
\end{align*}%

\luca{The first main question that we answer in this paper
concerns then the well-posedness of system \eqref{eq1:intro}--\eqref{init:intro}
in the case of arbitrarily singular nonlinearities $\beta$ and $\psi$.
Our first main result (see Theorem~\ref{thm-prelim}) is a proof of
\pier{the} existence and uniqueness of {\em bounded} solutions 
for the system \eqref{eq1:intro}--\eqref{init:intro}
under no growth assumptions on $\beta$ and $\psi$, 
possibly including logarithmic behaviours as above.
In this direction, we are inspired by some arguments performed in
\cite{BCT}, covering the analysis of the system \eqref{eq1:intro}--\eqref{init:intro}
in the singular case $\delta=0$.
The main idea here 
was based on the fact that if the initial condition is within a finite interval
(contained in the effective domain of the potential $\psi$)
and if the bulk free energy has sufficiently fast growth, 
then the concentration is essentially bounded in the parabolic domain 
$Q_T=\pier{\Omega \times (0,T)} $. This allowed to deduce, 
through the Gronwall lemma, a contraction estimate to prove existence and uniqueness of solutions. 
However, in our case
the presence of the term $-\delta\Delta u$ in the inclusion for the chemical potential
prevents us from relying on a similar contraction argument.
To overcome this problem, we prove
a preliminary boundedness result:
using a maximum principle for doubly nonlinear parabolic equations 
in combination with a suitable time-discretization of the problem, 
we show that the solution $u$ 
never touches the edges of the domain of $\psi$
and remains bounded in the parabolic domain $Q_T$.
\gius{Thus, we are able} to prove well-posedness also with very singular behaviours of $\psi$ and $\beta$,
under less stringent conditions on the potential than those in \cite{bcst1}.
\gius{This novel result actually improves the previous approaches to the problem; moreover, the argument is not standard at all and, in our opinion, gives value to our contribution.}}

\luca{\pier{Once} well-posedness is established in this general framework, 
we focus on questions of more qualitative nature.
More specifically,} 
both the viscous term and the energetic term in \eqref{eq2:intro} 
provide assistance in handling the possible non-smoothness of $\beta$ 
and the nonlinearity of $\psi'$. 
It is then natural to inquire whether one of these terms, alone, 
is sufficient to guarantee well-posedness, and whether the singular 
limits obtained when either $\varepsilon\searrow 0$ or $\delta\searrow 0$ converge to
the the limiting equations.

The second main result of this paper
(see Theorem \ref{thm1}) is an asymptotic result, 
and shows convergence of 
the solutions of \eqref{eq1:intro}--\eqref{init:intro} 
in the limit $\varepsilon\searrow 0$, \luca{with $\delta>0$ being fixed.
This confirms that the diffusive regularization $-\delta\Delta u$ alone
allows to handle the doubly nonlinear problem, even when the 
nonlinearity $\beta$ acting on the viscosity is multivalued and 
not necessarily coercive. For example, 
a physically relevant choice for $\beta$
in connection \pier{with} phase-change and Stefan-type problems is
the multivalued graph
\[
\beta_{\pier{sign}}(r):=\begin{cases}
-1 \quad&\text{if } r<0\,,\\
[-1,1] \quad&\text{if } r=0\,,\\
1 \quad&\text{if } r>0\,.
\end{cases}
\]
Note that although 
$\beta_{sign}$ \pier{is} nonsmooth and noncoercive, it can
be chosen in the equation \eqref{eq2:intro} as long as $\delta>0$ only
(even for $\eps=0$).
From the mathematical perspective, 
the main tools that we use here are 
compactness arguments combined with 
monotone analysis techniques in order to
pass to the limit in the two nonlinearities.}

\luca{
An alternative scenario to handle the monotone term would be 
to accompany it with the viscous regularization $\eps\partial_t u$ alone,
discarding the energetic regularization $-\delta\Delta u$ through the interface energy.}
The degenerate case $\delta=0$ was the object of the investigation in \cite{BCT}.
This belongs to a wider class of degenerate parabolic systems
which find their application in the modelling of hysteretic behaviour 
in diffusion process, such as hysteresis in porous media 
\cite{Bagagiolo2000,Botkin2016,Schweizer2017, Tomas} or in hydrogen storage devices \cite{Latroce}. 
In all these cases, the major manifestation of hysteresis is in the fact 
that the pressure that is needed to induce adsorption is higher than 
the pressure needed to induce desorption. 
\luca{This scenario is the object of our third Theorem \ref{thm2}, 
which covers the asymptotics of the system \eqref{eq1:intro}--\eqref{eq2:intro}
as $\delta\searrow0$, with $\eps>0$ being fixed.
The main tools that we rely on consist again \pier{in}
compactness and monotonicity techniques: furthermore, 
in the asymptotics $\delta\searrow0$ we are able to show
some refined $L^\infty$-estimates, allowing us to 
prove also the convergence rate as $\delta\searrow0$.}

\luca{Note that if in addition to $\delta=0$ we assume also $\beta=0$, 
then we recover the viscous forward-backward parabolic equation studied in \cite{NovicP1991TAMS}.
The asymptotics $\delta\searrow0$ in the viscous case $\eps>0$ and with $\beta=0$
was studied in the work \cite{colli-scar16},
where convergence of the vioscous Cahn\pier{--}Hilliard to the limiting
forward-backward parabolic equation was proved.}

Here is the outline of the paper.
In the next section we state the precise assumptions, the analytical setting, and the main theorems that 
we prove. In Section~\ref{sec:prelim}, we prove the existence result for 
$\delta,\varepsilon>0$ generalizing the results in \cite{bcst1}. 
Then in Sections~\ref{sec:limit-as-epssearrow0} 
and \ref{sec:delta-to-0} 
we perform the asymptotics investigation 
once we let vanish the approximating paramaters $\varepsilon$ and $\delta$, respectively.


\section{Assumptions and main results}
\setcounter{equation}{0}
\label{results}

Throughout the paper, $\Omega$ is a smooth bounded domain in $\erre^3$ with  boundary $\Gamma$
and $T>0$ is a fixed final time; for any $t\in(0,T]$ we use the notation
\[
  Q_t:=\Omega\times(0,t)\,, \quad \Sigma_t:=\Gamma\times(0,t)\,, \quad Q:=Q_T\,, \quad \Sigma:=\Sigma_T.
\]
Moreover, we introduce the spaces
\begin{align*}
  &H:=L^2(\Omega)\,, \quad V:=H^1(\Omega)\,, \quad V_0:=H^1_0(\Omega)\,,\\
  &W= H^2(\Omega)\,, \quad W_0=W\cap V_0\,, \quad
  W_{\bf n}:=\left\{v\in W:\, \partial_{\bf n}v=0 \, \text{ a.e.~on } \Gamma\right\}\,,
\end{align*}
endowed with their usual norms, and we identify $H$ with its dual,
so that $(V, H, V^*)$ is a Hilbert triplet.
The symbol $\left<\cdot, \cdot\right>$ denotes the duality pairing between $V^*$ and $V$.
We will need the following lemma, which is a variation of 
the well-know compactness lemma (see e.g.~\cite[Lem.~5.1, p.~58]{lions}).

\begin{lem}
  For every $\sigma>0$, there exists $C_\sigma>0$ such that 
  \beq\label{comp}
  \norm{z}^2_H\leq \sigma\norm{\nabla z}^2_H + C_\sigma\norm{z}^2_{V_0^*} \quad\forall\,z\in V\,.
  \eeq
\end{lem}

\begin{proof}
  By contradiction, assume that there is
  $\bar\sigma>0$ and a sequence $(z_n)_n\subseteq V$ such that 
  \[
  \norm{z_n}^2_H > \bar\sigma\norm{\nabla z_n}^2_H + n\norm{z_n}^2_{V_0^*} \quad\forall\,n\in\enne\,.
  \]
  Then, setting $v_n:=z_n/\norm{z_n}_H$ (note that $z_n\neq0$ for all $n$), it follows immediately that
  \[
  \norm{v_n}_H=1\,, \quad
  \bar\sigma\norm{\nabla v_n}^2_H + n\norm{v_n}^2_{V_0^*} < 1 \qquad\forall\,n\in\enne\,.
  \]
  Consequently, we deduce that there is $v\in H$ and $w\in H^N$ such that, as $n\to\infty$,
  \[
  v_n\wto v\quad\text{in } H\,, \qquad
  \nabla v_n \wto w \quad\text{in } H^N\,, \qquad
  v_n\to 0 \quad\text{in } V_0^*\,.
  \]
  The first two convergences imply that $v\in V$, $w=\nabla v$ and $v_n\wto v$ in $V$.
  Since $V\cembed H$ is compact,
  we deduce that $v_n\to v$ in $H$.
  Moreover, from the third convergence and the fact that $H\embed V_0^*$ continuously,
  we infer that $v=0$. However, by the strong convergence in $H$
  we have 
  \[
  0=\norm{v}_H=\lim_{n\to\infty}\norm{v_n}_H=1\,,
  \]
  which is absurd. This concludes the proof.
\end{proof}

We assume that
\begin{gather}
  \label{psi1}
  \psi\in C^2(a, b)\,, \quad -\infty\leq a<b\leq+\infty\,,\\
  \label{psi3}
  \psi(r)\geq0 \quad\forall\, r\in(a,b)\,,\\
  \label{psi4}
  \lim_{r\to a^+}\psi'(r)=-\infty\,, \quad \lim_{r\to b^-}\psi'(r)=+\infty\,,\\
  \label{psi5}
  \psi''(r)> -K \quad\forall\,r\in(a,b)\pier{,}
\end{gather}
for a positive constant $K$. It is convenient to introduce 
\beq
\label{pier1}
  \gamma:(a,b)\to\erre\,, \qquad \gamma(r):=\psi'(r) + Kr\,, \quad r\in\erre\,,
\eeq
which is maximal monotone and strictly increasing.
In particular, there exists a unique $r_0\in(a,b)$ such that $\gamma(r_0)=0$.
We also define the proper convex function
\beq
  \label{pier2}
  \widehat\gamma(r):=\int_{r_0}^r\gamma(s)\,ds\,, \qquad r\in(a,b)\,.
\eeq 
Furthermore, let
\beq
  \label{beta_hat}
  \widehat{\beta}:\erre\to[0,+\infty] \quad\text{convex and l.s.c., with}\quad
  \widehat{\beta}(0)=0\,, \ \quad\beta:=\partial\widehat\beta\,,
\eeq
and note that $0\in\beta(0)$. \luca{We shall denote the convex conjugate of $\widehat\beta$
by $\widehat{\beta^{-1}}$. Note that 
$\widehat{\beta^{-1}}:\erre\to[0,+\infty]$ with $\widehat{\beta^{-1}}(0)=0$, and 
$\partial\widehat{\beta^{-1}}$ \pier{is nothing but $\beta^{-1}$, the inverse graph of $\beta$}. Let us also recall
the Young inequality:
\[
  rs\leq\widehat\beta(r) + \widehat{\beta^{-1}}(s) \quad\forall\,r,s\in\erre\,, \quad
  \text{\pier{where the equality} holds if and only if } s\in\beta(r)\,.
\]
For general results on convex analysis 
we refer to \cite{barbu-monot}.}

In this setting, existence of solution for problem \eqref{eq1:intro}--\eqref{init:intro} has been shown 
in \cite{bcst1}
for $\eps,\delta>0$ fixed, with additional growth restrictions
either on $\beta$ or $\psi$.
The first main theorem that we prove here is a generalized existence result
for the problem \eqref{eq1:intro}--\eqref{init:intro} with
$\eps,\delta>0$ fixed under no growth restrictions on the operators.
\begin{thm}\label{thm-prelim}
  Let $\eps>0$, $\delta>0$, and
  \begin{gather}
    \label{pier3}
     u_{0,\eps\delta}\in \pier{W_{\bf n}}\,,\qquad  
     \exists\,[a_0, b_0]\subset(a,b):\; a_0\leq u_{0,\eps\delta} \leq b_0 \quad\text{a.e.~in } \Omega\,,\\
     \label{pier4}
     g_{\eps\delta}\in H^1(0,T; H)\cap L^\infty(Q)\,.
  \end{gather}
  Then, there are
  two constants $a_0', b_0'\in\erre$, possibly depending on
  $\eps$ and $\delta$, with $[a_0,b_0]\subseteq[a_0',b_0']\subset(a,b)$,
  and a unique triplet $(u_{\eps\delta},\mu_{\eps\delta},\xi_{\eps\delta})$
  such that 
  \begin{gather}
    \label{u}
    u_{\eps\delta}\in W^{1,\infty}(0,T; H)\cap H^1(0,T; V)\cap 
    L^\infty(0,T;\pier{W_{\bf n}})\\
    \label{u-inf}
    a_0'\leq u_{\eps\delta} \leq b_0' \quad\text{a.e.~ in } Q\,,\\
    \label{mu}
    \mu_{\eps\delta}\in L^\infty(0,T; W_0)\cap L^2(0,T; H^3(\Omega))\,,\\
    \label{xi_psi}
    \xi_{\eps\delta} \in L^\infty(0,T; H)\,, \quad \ \psi'(u)\in L^\infty(Q)\,,\\
    \label{incl}
    \xi_{\eps\delta}\in\beta(\partial_t u_{\eps\delta}) \quad\text{a.e.~in } Q\,,\\
    \label{1}
    \partial_t u_{\eps\delta}(t)-\Delta \mu_{\eps\delta}(t) = 0 \quad\text{for a.e.~}t\in(0,T)\,,\\
    \label{2}
    \mu_{\eps\delta}(t)=\eps\partial_t u_{\eps\delta}(t)+\xi_{\eps\delta}(t)-\delta\Delta u_{\eps\delta}(t) +
    \psi'(u_{\eps\delta}(t)) + g(t)\quad\text{for a.e.~}t\in(0,T)\,,\\
    \label{3}
    u_{\eps\delta}(0)=u_0.
  \end{gather}
\end{thm}

\pier{A continuous dependence result follows then.}

\begin{thm}
  \label{contdep} 
Let $\eps>0$ and $\delta >0$.
For any sets of data $(u_{0,i}, g_i)$, $i=1,2,$ satisfying \eqref{pier3}--\eqref{pier4},
let $(u_i, \mu_i, \xi_i)$
denote any corresponding solutions to \eqref{u}--\eqref{3}.
Then, there exists a constant $C_{\eps\delta}$, depending on
the data, such that 
  \begin{align}
   \label{dipcont}
    &\| \mu_1 -\mu_2 \|_{L^2(0,T;V_0)}^2 +   \| u_1 - u_2 \|_{H^1(0,T;H) \cap L^\infty (0,T;V)} ^2
    + \int_Q (\xi_1 - \xi_2 )(\partial_t u_1 - \partial_t u_2) \nonumber \\ 
    &\leq C_{\eps\delta} \left( \| u_{0,1} - u_{0,2} \|_V^2+ \| g_1 -g_2 \|_{L^2(0,T;H)}^2  \right). 
  \end{align}
\end{thm}

\pier{At this point, we state our first asymptotic result, keeping $\delta >0$ fixed 
and letting $\eps$ tend to $0$.}

\begin{thm}
  \label{thm1}
  Let $\delta>0$ be fixed and assume that
  \begin{gather}
  \label{ip_eps1}
  u_0 \in \pier{W_{\bf n}}\,, \quad \psi'(u_0)\in H\,, \quad g\in H^1(0,T; H)\,,\quad g(0)\in L^\infty(\Omega)\,,\\
  \label{ip_eps2}
  \exists\,C_1,C_2>0:\quad \psi(r)\geq C_1|r|^2 - C_2 \quad\forall\,r\in D(\psi)\,,\\
  \label{ip_eps3}
  z_0:=-\delta \Delta u_0 + \psi'(u_0) + g(0) \quad\text{is such that}\quad
  \widehat{\beta^{-1}}(-z_0) \in L^1(\Omega)\,,\\
  \label{ip_psi}
  (a,b)=\erre\,, \qquad\exists\,M>0:\quad |\psi''(r)|\leq M(1+|r|^5) \quad\forall\,r\in\erre\,.
  \end{gather}
  Let also $(g_\eps)_\eps\subset H^1(0,T; H)\cap L^\infty(Q)$ fulfill  
  \beq
    \label{conv_g_eps}
    g_\eps(0)=g(0)\,, \qquad
    g_\eps\to g \quad\text{in } H^1(0,T; H)\,.
  \eeq
  Then, if $(u_{\eps},\mu_{\eps},\xi_{\eps})_{\eps>0}$ denotes the unique family
  solving \eqref{u}--\eqref{3} with respect to the data
  $(u_0,g_\eps)$, there exists a triplet $(u,\mu,\xi)$ such that 
  \begin{gather}
    \label{u_lim}
    u \in W^{1,\infty}(0,T; V_0^*)\cap H^1(0,T; V) \cap L^\infty(0,T; W_{\bf n})\,,\\
    \mu \in L^\infty(0,T; V_0)\cap L^2(0,T; H^3(\Omega))\,, \\
    \label{xi_psi_lim}
    \xi \in L^\infty(0,T; H)\,, \qquad \psi'(u)\in L^\infty(0,T; H)\,,\\
    \label{pier6}
    \xi\in\beta(\partial_t u) \quad\text{a.e.~in } Q\,,\\
    \label{1lim}
    \partial_t u(t) - \Delta\mu(t) = 0 \quad\text{for a.e.~}t\in(0,T)\,,\\
    \label{2lim}
    \mu(t)=\xi(t)-\delta\Delta u(t) + \psi'(u(t)) + g(t) \quad\text{for a.e.~}t\in(0,T)\,,\\
    u(0)=u_0\,.
  \end{gather}
  and a sequence $(\eps_n)_n$ such that, as $n\to\infty$, $\eps_n\searrow0$ and
  \begin{gather}
  \label{conv1}
  u_{\eps_n} \wstarto u \quad\text{in } W^{1,\infty}(0,T; V_0^*)\cap L^\infty(0,T; W_{\bf n})\,,\qquad
  u_{\eps_n}\wto u \quad\text{in } H^1(0,T; V)\,,\\
  \mu_{\eps_n} \wstarto\mu \quad\text{in } L^\infty(0,T; V_0)\,,\qquad
  \mu_{\eps_n} \wto \mu \quad\text{in } L^2(0,T; H^3(\Omega))\,,\\
  \label{conv3}
  \xi_{\eps_n} \wstarto \xi \quad\text{in } L^\infty(0,T; H)\,, \qquad 
  \psi'(u_{\eps_n})\wstarto \psi'(u) \quad\text{in } L^\infty(0,T; H)\,,\\
  \eps_n \partial_t u_{\eps_n} \to 0 \quad\text{in } L^\infty(0,T; H)\,.
  \end{gather}
  Furthermore, if instead of \eqref{ip_psi} 
  we assume that
  \beq
  \label{ip_beta}
    D(\beta)=\erre\,, \qquad\exists M>0:\quad |s|\leq M(1+|r|) \quad\forall\,r\in\erre\,,\quad\forall\,s\in\beta(r)\,,
  \eeq
  then the same
  conclusion is true replacing $L^\infty$ with $L^2$ in \eqref{u_lim}, \eqref{xi_psi_lim},
  \eqref{conv1} and \eqref{conv3}.
\end{thm}

\begin{rmk}\label{rmk_eps}
  Let us comment on the construction of a possible family $(g_\eps)_\eps$
  satisfying \eqref{conv_g_eps}. Since $g(0)\in L^\infty(\Omega)$, 
  for instance one can choose
  $g_\eps:=T_\eps(g)$, where $T_\eps:\erre\to\erre$ is the usual
  truncation operator at level $1/\eps$, i.e., $T_\eps(r):=\max\{\min\{r,\pier{1/\eps\},-1/\eps}\}$ for $r\in\erre$.
  Indeed, it is not difficult to check that $g_\eps(0)=g(0)$ provided that $\frac1\eps>\norm{g(0)}_{L^\infty(\Omega)}$
  and $g_\eps\to g$ in $H^1(0,T; H)$.
\end{rmk}

\pier{The second asymptotic result investigates the behavior of the system as $\delta \searrow 0$. In this case, we can prove the convergence of the whole sequence and even an error estimate in terms of $\delta$ (see~\eqref{errorest}).}

\begin{thm}
  \label{thm2}
 Let $\eps>0$ be fixed. Assume
 \begin{gather}
  \label{ip_delta1}
  u_0 \in H\,, \qquad \exists\,[a_0,b_0]\subseteq(a,b):\;a_0\leq u_0\leq b_0 \quad\text{a.e.~in } \Omega\,,\\
  \label{ip_delta2}
  g\in H^1(0,T; H)\cap L^\infty(Q)\,.
  \end{gather}
  Let $(u_{0\delta})_\delta\subset \pier{W_{\bf n}}$ and 
  $(g_\delta)_\delta\subset H^1(0,T; H)\cap L^2(0,T; V)\cap L^\infty(Q)$ be such that
  \begin{gather}
  \label{ip_delta3}
  a_0\leq u_{0\delta}\leq b_0 \quad\text{a.e.~in } \Omega\,,\\
  \label{ip_delta4}
  \delta^{1/2}\norm{\nabla u_{0\delta}}^2
  +\delta^{3/2}\norm{\Delta u_{0\delta}}_H^2 
  +\norm{\psi'(u_{0\delta})}_H^2\leq C\,,\\ 
  \label{ip_delta5}
  \norm{g_\delta}_{H^1(0,T; H)\cap L^\infty(Q)} + \delta^{1/2}\norm{g_\delta}^2_{L^2(0,T; V)}\leq C\,,\\
  \label{ip_delta6}
  u_{0,\delta} \to u_0 \quad\text{in } H\,, \qquad
  g_\delta \to g \quad\text{in } H^1(0,T; H)\,.
  \end{gather}
  Then, if $(u_{\delta},\mu_{\delta},\xi_{\delta})_{\delta>0}$ denotes the unique family
  solving \eqref{u}--\eqref{3} with respect to the data $(u_{0\delta}, g_\delta)$, 
  there \pier{exist} a triplet $(u,\mu,\xi)$ \pier{and an interval} $[a_0', b_0']\subset(a,b)$ such that
\begin{gather}
  \label{u_lim'}
  u \in W^{1,\infty}(0,T; H)\cap L^\infty(Q)\,,\qquad
  a_0'\leq u \leq b_0' \quad\text{a.e.~in } Q\,,\\
  \label{mu_lim'}
  \mu \in L^\infty(0,T; W_0)\,,\\
  \label{xi_lim'}
  \xi \in L^\infty(0,T; H)\,, \qquad \psi'(u) \in L^\infty(Q)\,,\\
  \xi \in \beta(\partial_t u) \qquad\text{a.e.~in } Q\,,\\
  \label{1lim'}
  \partial_t u(t) - \Delta\mu(t) =0  \quad\text{for a.e.~}t\in(0,T)\,,\\
  \label{2lim'}
  \mu(t)=\eps\partial_t u(t) + \xi(t) + \psi'(u(t)) - g(t) \quad\text{for a.e.~}t\in(0,T)\,,\\
  u(0)=u_0\,
\end{gather}
and, as $\delta\searrow0$,
\begin{gather*}
  u_\delta \wstarto u \quad\text{in } W^{1,\infty}(0,T; H)\cap L^\infty(Q)\,,\qquad
  u_\delta \to u \quad\text{in } H^1(0,T; H)\,,\\
  \mu_\delta \wstarto \mu \quad\text{in } L^\infty(0,T; W_0)\,, \qquad
  \mu_\delta\to \mu \quad\text{in } L^2(0,T; V_0)\,,\\
  \psi'(u_\delta) \wstarto \psi'(u) \quad\text{in } L^\infty(Q)\,, \qquad
  \psi'(u_\delta) \to \psi'(u) \quad\text{in } L^2(0,T; H)\,,\\
  \xi_\delta \wstarto\xi \quad\text{in } L^\infty(0,T; H)\,,\\
  \delta^{1/2}u\to 0 \quad\text{in } H^1(0,T; V)\,, \qquad
  \delta u_\delta \to 0 \quad\text{in } L^\infty(0,T; W_{\bf n})\,.
\end{gather*}
In particular, there exists a constant $M>0$, independent of $\delta$, such that 
\begin{align}
\label{errorest}
  &\norm{\mu_\delta-\mu}_{L^2(0,T; V_0)} + 
  \norm{u_\delta-u}_{H^1(0,T; H)}\nonumber \\
  &\qquad\leq
  M\left(\delta^{1/4} + \norm{u_{0\delta}-u_0}_H
  +\norm{g_\delta-g}_{L^2(0,T; H)}\right)\,.
\end{align}
\end{thm}

\begin{rmk}
\pier{Note that the limit problem with $\delta=0$ admits a unique solution, as it is proved in  \cite[Theorem~2.1]{BCT}. This result, and in particular \cite[estimate~(2.9)]{BCT}, are related to the error estimate \eqref{errorest} stated here and can be compared with 
the continuous dependence estimate \eqref{dipcont} for $\varepsilon,\delta>0$. 
Actually, we point out that here\pier{,
in order to prove Theorem~\ref{contdep},} we are using some stronger assumptions on the initial datum depending on the fact that we deal with spatial regularity for 
$\delta>0$.}
\end{rmk}

\begin{rmk}
  Let us show that, under the assumptions \eqref{ip_delta1}--\eqref{ip_delta2},
  two sequences $(u_{0\delta})_\delta$ and $(g_\delta)_\delta$
  with the properties above always exist. \pier{Specifically, to construct them
  it is possible to employ} a singular perturbation technique. 
  \luca{Indeed, we could introduce the solution $u_{0\delta}$ of the \pier{elliptic problem}
  \begin{equation}\label{ell1}
    \begin{cases}
  u_{0\delta}-\delta^{1/2}\Delta u_{0\delta}=u_0 \quad&\text{in } \Omega\,,\\
  \partial_{\bf n}u_{0\delta}=0 \quad&\text{\pier{on} } \Gamma
  \end{cases}
  \end{equation}
  and \pier{let $g_\delta$ be} the solution of
  \begin{equation}
  \pier{\begin{cases}
   g_\delta (t) -\delta^{1/2}\Delta g_\delta (t)=g (t) \quad&\text{in } \Omega\,,\\
  \partial_{\bf n}g_\delta (t) = 0 \quad&\text{\pier{on} } \Gamma\,
  \end{cases}}\label{ell2}
  \end{equation}}%
\pier{for all $t\in [0,T].$ Then, \eqref{ip_delta3} follows from \eqref{ip_delta1} and the maximum principle, while \eqref{ip_delta4} can be shown by testing the equation in \eqref{ell1} by $u_{0\delta}$ and subsequently comparing the terms and recalling the assumption \eqref{psi1}. Also, the verification of  \eqref{ip_delta5} and \eqref{ip_delta6} is not difficult, in particular for \eqref{ip_delta6} one can take advantage of the properties
$$ \limsup_{\delta \searrow 0} \norm{u_{0\delta}}_{H}^2 \leq \norm{u_{0}}_{H}^2 , 
\quad \ \limsup_{\delta \searrow 0} \norm{g_\delta}_{H^1(0,T; H)}^2 \leq \norm{g}_{H^1(0,T; H)}^2.
$$}
\end{rmk}

\begin{rmk}
The regularities $u_0\in V$ and $g\in L^2(0,T; V)$ imply $u\in H^1(0,T; V)$
also for $\delta=0$. Indeed, as it is discussed in 
in \cite[Remark~5.1]{BCT} we can formally take the gradient of the equation \eqref{2lim'}
and test it by $\partial_t u$: using the Lipschitz continuity of the operator 
\pier{$(I+\beta)^{-1}$ (where $I$ denotes the identity) and the Gronwall lemma, 
it is straightforward to infer that $ u\in H^1(0,T; V)$} (see~\cite[Remark~5.1]{BCT} for details).
\end{rmk}


\section{Proof of Theorems~\ref{thm-prelim}--\ref{contdep}}\label{sec:prelim}
\setcounter{equation}{0}

This section is devoted to the proof of the above mentioned results.

\subsection{The existence result}
We focus here on the proof of Theorem~\ref{thm-prelim}.
The main idea is to approximate the problem as in \cite{bcst1} and to show 
that the approximated solutions satisfy further refined uniform estimates.
As $\delta$ and $\eps$ are fixed positive numbers in this section, 
we shall consider with no restriction that $\eps=\delta=1$. Moreover,
in order to simplify the presentation,
we shall avoid the subscripts $\eps$ and $\delta$ for $g$ and $u_0$.

Let now $(g_\lambda)_{\lambda\in(0,1)}\subseteq  H^1(0,T; H)\cap L^2(0,T; V)\cap L^\infty(Q)$ such that
\[
  g_\lambda \to g \quad\text{in } H^1(0,T; H) \quad\text{as } \lambda\searrow0\,,
  \qquad 
  \norm{g_\lambda}_{L^\infty(Q)}\leq\norm{g}_{L^\infty(Q)}\quad\forall\,\lambda\in(0,1)\,.
\]
For example, one can take \pier{(cf. \eqref{ell2})} $g_\lambda$ as the unique solution 
to the elliptic problem
\[
  \begin{cases}
  g_\lambda-\lambda\Delta g_\lambda = g \quad&\text{in } \Omega\,,\\
  \partial_{\bf n}g_\lambda = 0 \quad&\text{\pier{on} } \Gamma\,.
  \end{cases}
\]
Furthermore, denote by $T_\lambda:\erre\to\erre$ the truncation operator at level \pier{$1/\lambda$}, 
already defined in Remark~\ref{rmk_eps}.
Then, reasoning as in\cite{bcst1} we know that there exist a unique pair $(u_\lambda,\mu_\lambda)$ such that 
\begin{gather}
    \label{u_lambda}
    u_\lambda \in C^1([0,T]; H)\cap H^1(0,T; V)  \cap C^0([0,T]; W_{\bf n}) \cap L^2(0,T; H^3(\Omega)) \,,\\
    \label{mu_lambda}
      \mu_\lambda \in C^0([0,T]; W_0)\cap L^2(0,T; H^3(\Omega))
\end{gather}
and, for every $t\in[0,T]$,
\begin{align}
    \label{eq1_approx}
      &\partial_t u_\lambda(t)-\Delta\mu_\lambda(t)=0\,,\\ 
      \nonumber
      &\mu_\lambda(t)=\partial_t u_\lambda(t) +\beta_\lambda(\partial_t u_\lambda(t))-\Delta u_\lambda(t)
      			+\lambda u_\lambda(t) \\
			\label{eq2_approx}   
                         &\qquad\quad+\gamma_\lambda(u_\lambda(t)) 
                         -KT_\lambda(I+\lambda\gamma)^{-1}(u_\lambda(t))+ g_\lambda(t) \,,\\
       \label{eq3_approx}
       &u_\lambda(0)=u_0\,,                    
\end{align}
where $\gamma$ is defined in \eqref{pier1} and $\gamma_\lambda, \beta_\lambda$
denote the Yosida approximations of the maximal monotone graphs $\gamma$ and $\beta$,
respectively.
\luca{Note that \eqref{eq1_approx}--\eqref{eq3_approx} is indeed an approximation
of the original system \eqref{1}--\eqref{3} in the following sense.
The term $\lambda u_\lambda$ represents a (small)
elliptic regularization that is going to vanish \pier{as} $\lambda\searrow0$.
\pier{Moreover}, since $T_\lambda$ and $(I+\lambda\gamma)^{-1}$
converge to the identity in $(a,b)$, the contribution 
$-KT_\lambda(I+\lambda\gamma)^{-1}(u_\lambda)$ represents an approximation
of $-Ku$, hence the terms
$\gamma_\lambda(u_\lambda) 
-KT_\lambda(I+\lambda\gamma)^{-1}(u_\lambda)$ provide
an approximation of $\pier{\psi'}(u)$.}

\luca{The first estimates can be obtained with 
no additional effort from the arguments in \cite[\S~5.1--5.2]{bcst1}
and owing to the Lipschitz-continuity of 
$T_\lambda$ and $(I+\lambda\gamma)^{-1}$ on $\erre$.
In particular, we can test \eqref{eq1_approx} by $\mu_\lambda$, 
\eqref{eq2_approx} by $\partial_t u_\lambda$, and sum.
Secondly, we can also (formally)
test \eqref{eq1_approx} by $\partial_t \mu_\lambda$, 
the time derivative of \eqref{eq2_approx} by $\partial_t u_\lambda$, and sum.}
\pier{Then, by also comparing the terms in \eqref{eq1_approx} 
and using the elliptic regularity theory (as in \cite[\S~5.1--5.2]{bcst1}),} 
it is readily seen that 
\beq\label{est-lam}
  \norm{u_\lambda}_{W^{1,\infty}(0,T; H)\cap H^1 (0,T;V)} +
  \norm{\mu_\lambda}_{L^\infty(0,T; W_0)\cap L^2(0,T; H^3(\Omega))}\leq c
\eeq
for a positive constant $c$, independent of $\lambda$.

We show now that $u_\lambda$ satisfies also an $L^\infty$-estimate
by proving a maximum principle
that arises from a time-discretization of the approximated problem.
We shall need the following result, \pier{for which we refer to}~\cite[Prop.~11.6]{roub}.
\begin{prop}\label{roub}
Let $\Phi:V\to[0,+\infty]$ and $\Xi:H\to\erre$ be proper, convex,
lower semicontinuous, and assume
that there exist $c_0, c_1, c_2>0$ such that
\begin{align*}
  &\ip{w}{v}\geq c_0\norm{v}_H^2 &&\forall\,v\in H\,,\quad\forall\,w\in\partial\Xi(v)\,,\\
  &\norm{w}_H\leq c_1(1+\norm{v}_H)
  &&\forall\,v\in H\,,\quad\forall\,w\in\partial\Xi(v)\,,\\
  &\ip{w}{v}\geq c_2\norm{v}_V^2 &&\forall\,v\in V\,,\quad\forall\,w\in\partial\Phi(v)\,.
\end{align*}
Set $A_1:=\partial\Phi$, let $A_2:H\to H$ be Lipschitz-continuous and define $A:=A_1+A_2$.
Moreover, let $f\in L^2(0,T; H)$ and $v_0\in V\cap D(\Phi)$. 
For every $N\in\enne$ sufficiently large, we set $\tau:=T/N$ and consider the discretized problem
\beq\label{rothe-eq}
  \partial\Xi\left(\frac{v_\tau^{k}-v_\tau^{k-1}}{\tau}\right) + A(v_\tau^k)\ni f^k\,,
  \quad k=1,\ldots,N\,, \qquad v_\tau^0=v_0\,,
\eeq
with
\[
  f^k=\frac1\tau\int_{(k-1)\tau}^{k\tau}f(s)\,ds\,, \quad k=1,\ldots,N\,.
\]
Then, problem \eqref{rothe-eq}
admits a solution $(v_\tau^k)_{k=0,\ldots,N}$, and the piecewise affine
interpolants $v_\tau$ of $(v_\tau^k)_{k=0,\ldots,N}$ satisfy 
\[
  \norm{v_\tau}_{H^1(0,T; H)\cap L^\infty(0,T; V)}\leq c
\]
for a positive constant $c$ independent of $\tau$. Furthermore, there are a subsequence
$(\tau_i)_{i\in\enne}$, with $\tau_i\to 0$
and an element $v\in H^1(0,T; H)\cap L^\infty(0,T; V)$,
such that $v_{\tau_i}\wstarto v$ in $H^1(0,T; H)\cap L^\infty(0,T; V)$
and $v$ is a solution to the problem
\[
  \partial\Xi\left(\partial_t v\right) + A(v)\ni f\,, \qquad v(0)=v_0\,.
\]
\end{prop}

Now, note that equation \eqref{eq2_approx} can be written as
\beq\label{eq:app-lam}
  (I + \beta_\lambda)(\partial_t u_\lambda) + 
  (-\Delta + \gamma_\lambda + \lambda I - KT_\lambda(I+\lambda\gamma)^{-1})(u_\lambda) 
  = \mu_\lambda-g_\lambda\,.
\eeq
Hence, for any $\lambda\in(0,1)$ fixed, we can apply
Proposition~\ref{roub} with the choices
\begin{align*}
  &\Xi(v):=\frac12\norm{v}_H^2 + \int_\Omega\widehat\beta_\lambda(v)\,, \quad v\in H\,,\qquad
  \Phi(v):=\frac12\int_\Omega\left(|\nabla v|^2 + \lambda|v^2|\right)\,,\quad v\in V\,,\\
  &A_1:=-\Delta + \lambda I\,, \quad A_2:=\gamma_\lambda - KT_\lambda(I+\lambda\gamma)^{-1}\,, 
  \quad f:=\mu_\lambda-g_\lambda\,, \quad v_0:=u_{0,\eps\delta}\,.
\end{align*}
Let then $(u_{\lambda,\tau}^k)_{k=0,\ldots,N}$ be a Rothe-sequence for the approximated
problem with parameter $\lambda$. Then, 
since the solution $u_\lambda$ to \eqref{eq1_approx}--\eqref{eq3_approx}
is uniquely determined, 
setting $u_{\lambda,\tau}$ as
the piecewise affine interpolant of $(u_{\lambda,\tau}^k)_{k=0,\ldots,N}$, 
it turns out that
\beq
  \label{rothe-conv}
  u_{\lambda,\tau} \wstarto u_\lambda \quad\text{in } H^1(0,T; H)\cap L^\infty(0,T; V)
\eeq
for the whole sequence $(u_{\lambda,\tau})_\tau$.

Thanks to the estimate on $(\mu_\lambda)_\lambda$ and the boundedness of $(g_\lambda)_\lambda$, 
there exists a positive constant $M$, independent of $\lambda$, such that
\beq\label{est_inf}
 \|\mu_\lambda-g_\lambda\|_{L^\infty(Q)}\leq M.
\eeq
By the growth assumption on $\psi'$, there are $\bar a, \bar b\in \erre$
with $r_0\in(\bar a,\bar b)$, $[a_0,b_0]\subseteq[\bar a, \bar b]\subset(a,b)$, and
\begin{gather}
  {\psi'(r)\ge M+1}\quad\textrm{for all }r\in[\bar b,b)\label{eq:33b}\,,\\
  {\psi'(r)\le -M-1}\quad\textrm{for all }r\in (a, \bar a]\label{eq:33c}\,.
\end{gather}
Setting now $a_0':=\bar a - \frac{\bar a-a}{2}$ and $b_0':= \bar b + \frac{b-\bar b}{2}$,
we have $[a_0,b_0]\subseteq[\bar a, \bar b]\subset[a_0', b_0']\subset(a,b)$.
By the properties of the resolvent $(I+\lambda\gamma)^{-1}:\erre\to\erre$, it is well known that
\begin{equation}\label{eq:limit1}
\lim_{\lambda\searrow0}
(I+\lambda\gamma)^{-1}(a_0')=a_0'\,, \qquad
\lim_{\lambda\searrow0}
(I+\lambda\gamma)^{-1}(b_0')= b_0'\,.
\end{equation}
Note also that, since $\gamma(r_0)=0$, \pier{it holds} $(I+\lambda\gamma)^{-1}(r_0)=r_0$,
hence, recalling that $(I+\lambda\gamma)^{-1}$ is $1$-Lipschitz-continuous,
\begin{align*}
  |(I+\lambda\gamma)^{-1}(a_0')-r_0|=|(I+\lambda\gamma)^{-1}(a_0')-(I+\lambda\gamma)^{-1}(r_0)|
  \leq|a_0'-r_0|\,,\\
  |(I+\lambda\gamma)^{-1}(b_0')-r_0|=|(I+\lambda\gamma)^{-1}(b_0')-(I+\lambda\gamma)^{-1}(r_0)|
  \leq|b_0'-r_0|\,.
\end{align*}
Since $r_0\in(a_0',b_0')$, we deduce from the last \pier{inequalities} that 
$(I+\lambda\gamma)^{-1}(a_0')\ge a_0'$ and 
$(I+\lambda\gamma)^{-1}(b_0')\le b_0'$. Then, by making use of \eqref{eq:limit1}, we conclude
that 
there exists $\lambda_0\in(0,1)$ such that, for every $\lambda\in(0,\lambda_0)$,
\[
  a_0' \leq (I+\lambda\gamma)^{-1}(a_0') \leq \bar a\,, \qquad
  \bar b \leq (I+\lambda\gamma)^{-1}(b_0') \leq b_0'\,.
\]
Moreover, since the resolvent $(I+\lambda\gamma)^{-1}$
is non-decreasing, for every $\lambda\in(0,\lambda_0)$ we have
\beq\label{aux_res}
(I+\lambda\gamma)^{-1}(r)\leq \bar a \quad\forall\, r\in(a, a_0']\,, \qquad
(I+\lambda\gamma)^{-1}(r)\geq \bar b \quad\forall\, r\in[b_0',b)\,.
\eeq
We claim now that if the initial datum $u_{0}$ satisfies
\[
a_0\leq u_{0} \leq b_0 \quad\text{a.e~in } \Omega\,,
\]
then
\beq\label{boundedness}
a_0'\leq u_\lambda \leq b_0' \quad\text{a.e~in } Q\,.
\eeq
Thanks to the convergence \eqref{rothe-conv}, it is enough to check that
\[
a_0'\leq u_{\lambda,\tau}^k \leq b_0' \quad\text{a.e.~in } \Omega\,,  \quad\text{for } k=0,\ldots,N\,.
\]
By contradiction, 
let $k$ be the smallest index such that 
$u_{\lambda,\tau}^k>b_0'$ on a set of positive measure in $\Omega$. 
Then, testing 
the analogue of
\eqref{rothe-eq} by $(u_{\lambda,\tau}^k- b_0')^+$ we have
\begin{align*}
  &\int_\Omega (I+\beta_\lambda)\left(\frac{u_{\lambda,\tau}^k-u_{\lambda,\tau}^{k-1}}{\tau}\right)(u_{\lambda,\tau}^k-b_0')^+
  +\int_\Omega |\nabla (u_{\lambda,\tau}^k-b_0')^+|^2\\
  &\qquad=\int_\Omega (\mu_\lambda-g_\lambda -\lambda u_{\lambda,\tau}^k
  -\gamma_\lambda(u_{\lambda,\tau}^k)
   + KT_\lambda(I+\lambda\gamma)^{-1}(u_{\lambda,\tau}^k))(u_{\lambda,\tau}^k- b_0')^+\\
   &\qquad=\int_{\{u_{\lambda,\tau}^k>b_0'\}} (\mu_\lambda-g_\lambda -\lambda u_{\lambda,\tau}^k
  -\gamma_\lambda(u_{\lambda,\tau}^k)
   + KT_\lambda(I+\lambda\gamma)^{-1}(u_{\lambda,\tau}^k))(u_{\lambda,\tau}^k- b_0')\,.
\end{align*}
Let us show that the right-hand side of the above equation is non positive if
\[
  \lambda <\min\left\{\frac1{|\bar b|}, \frac1{|b_0'|}\right\}\,,
\]
which is clearly not restrictive. Indeed, on the set $\{u_{\lambda,\tau}^k>b_0'\}$,
owing to \eqref{aux_res} we have \pier{that} $(I+\lambda\gamma)^{-1}(u_{\lambda,\tau}^k)\geq \bar b$ \pier{and consequently, as $\frac1\lambda>|\bar b|$, also that}
\[
T_\lambda(I+\lambda\gamma)^{-1}(u_{\lambda,\tau}^k) \leq (I+\lambda\gamma)^{-1}(u_{\lambda,\tau}^k)
\quad\text{a.e.~in } \{u_{\lambda,\tau}^k>b_0'\}\,.
\]
\luca{Recalling the definition of the Yosida approximation
$$\pier{\gamma_\lambda=\frac{I-(I+\lambda\gamma)^{-1}}{\lambda},}$$ 
we \pier{observe} that $\gamma_\lambda(r)=\gamma((I+\lambda\gamma)^{-1}(r))$
for every $r\in\erre$.}
Therefore, by \eqref{eq:33b} we infer that, on the set $\{u_{\lambda,\tau}^k>b_0'\}$,
\begin{align*}
  &\lambda u_{\lambda,\tau}^k+\gamma_\lambda(u_{\lambda,\tau}^k) - 
  KT_\lambda(I+\lambda\gamma)^{-1}(u_{\lambda,\tau}^k)\\
  &= \lambda u_{\lambda,\tau}^k + \gamma((I+\lambda\gamma)^{-1}(u_{\lambda,\tau}^k)) - 
  K(I+\lambda\gamma)^{-1}(u_{\lambda,\tau}^k)\\
  &\qquad+K(I+\lambda\gamma)^{-1}(u_{\lambda,\tau}^k)
  -KT_\lambda(I+\lambda\gamma)^{-1}(u_{\lambda,\tau}^k)\\
  &\geq \lambda u_{\lambda,\tau}^k + \psi'((I+\lambda\gamma)^{-1}(u_{\lambda,\tau}^k))\\
  &\geq \lambda b_0' + M +1 \geq M
\end{align*}
where we have used that $\lambda<\frac1{|b_0'|}$.

Hence, recalling \eqref{est_inf} we deduce that 
\[
  \int_{\{u_{\lambda,\tau}^k>b_0'\}} (\mu_\lambda-g_\lambda -\lambda u_{\lambda,\tau}^k
  -\gamma_\lambda(u_{\lambda,\tau}^k)
   + KT_\lambda(u_{\lambda,\tau}^k))(u_{\lambda,\tau}^k- b_0')\leq0\,.
\]
This implies that
\begin{equation}\label{ineq:max}
\int_{\{u_{\lambda,\tau}^k>b_0'\} }
(I+\beta_\lambda)\left(\frac{u_{\lambda,\tau}^k-u_{\lambda,\tau}^{k-1}}{\tau}\right)
(u_{\lambda,\tau}^k-b_0')\leq0\,.
\end{equation}
Now, on $\{u_{\lambda,\tau}^k>b_0'\}$ 
we must have $u_{\lambda,\tau}^k> b_0'\ge u_{\lambda,\tau}^{k-1}$
because of the definition of $k$. Thus, 
in view of the monotonicity of $\beta_\lambda$ and the fact that $\beta_\lambda(0)=0$,
the integrand in \eqref{ineq:max} is positive.
Since $\{u_{\lambda,\tau}^k>b_0'\}$ has positive measure by assumption this leads to a contradiction. 

The above argument implies that the Rothe approximation $u_{\lambda,\tau}^k$ satisfies the bound 
\[
\pier{u_{\lambda,\tau}}\leq b_0'\qquad \text{a.e. in } Q\,.
\]
A similar procedure can be used to prove that $\pier{u_{\lambda,\tau}} \geq a_0'$ 
a.e.~in $Q$ (for brevity we omit the details), hence \eqref{boundedness} follows.
Consequently, noting also that 
\begin{equation}
  a_0'\leq (I+\lambda\gamma)^{-1}(u_\lambda)\leq b_0' \quad\text{a.e.~in } Q
\label{pier5}
\end{equation} 
and $\gamma_\lambda(u_\lambda)=\gamma((I+\lambda\gamma)^{-1}(u_\lambda))$,
since $\gamma\in L^\infty(a_0',b_0')$ by \eqref{psi1} and \eqref{pier1}, we infer that
\beq
  \label{est-lam2}
  \norm{\gamma_\lambda(u_\lambda)}_{L^\infty(Q)}\leq c\,.
\eeq

Taking now the duality pairing between \eqref{eq2_approx} and $-\Delta\partial_t u_\lambda$,
integrating by parts we have
\begin{align*}
  &\int_{Q_t}|\nabla\partial_t u_\lambda|^2 
  + \int_{Q_t}\beta_\lambda'(\partial_t u_\lambda)|\nabla\partial_t u_\lambda|^2+
  \frac12\int_\Omega|\Delta u_\lambda(t)|^2
  +\frac\lambda2\int_\Omega|\nabla u_\lambda(t)|^2\\
  &= \frac12\int_\Omega\left(|\Delta u_0|^2 + \lambda|\nabla u_0|^2\right)
  +\int_{Q_t}\nabla\mu_\lambda\cdot\nabla\partial_t u_\lambda\\
  &-\int_{Q_t}\partial_t\left(g_\lambda+\gamma_\lambda(u_\lambda) 
  - KT_\lambda(I+\lambda\gamma)^{-1}(u_\lambda)\right)\Delta u_\lambda\\
  &+ \int_\Omega\left(g_\lambda+\gamma_\lambda(u_\lambda) 
  - KT_\lambda(I+\lambda\gamma)^{-1}(u_\lambda)\right)(t)\Delta u_\lambda(t)\\
  &-\int_\Omega\left(g_\lambda(0)+\gamma_\lambda(u_0) 
  - KT_\lambda(I+\lambda\gamma)^{-1}(u_0)\right)\Delta u_0\,.
\end{align*}
The first two terms on the right-hand side can be treated
by the assumptions on $u_0$ and the Young inequality.
About the third term, note that, 
since $\gamma'\in C^0([a_0',b_0'])$ by \eqref{psi1}, from \pier{\eqref{pier5}} 
it follows that
\[
|\partial_t \gamma_\lambda(u_\lambda)|=|\gamma_\lambda'(u_\lambda)\partial_t u_\lambda|
\leq \gamma'((I+\lambda\gamma)^{-1}(u_\lambda))|\partial_t u_\lambda|\leq c|\partial_t u_\lambda|\,.
\]
Hence, using the estimates \eqref{est-lam} and \eqref{est_inf}, as well as
the properties of $(g_\lambda)_\lambda$, again by the Young inequality we infer that
\begin{align*}
  \frac12\int_{Q_t}|\nabla\partial_t u_\lambda|^2 
  &+ \int_{Q_t}\beta_\lambda'(\partial_t u_\lambda)|\nabla\partial_t u_\lambda|^2+
  \frac14\int_\Omega|\Delta u_\lambda(t)|^2
  +\frac\lambda2\int_\Omega|\nabla u_\lambda(t)|^2\\
  &\leq
  c\left(1 + \int_{Q_t}|\Delta u_\lambda|^2\right)\,.
\end{align*}
The Gronwall lemma yields then
\beq
  \label{est-lam3}
  \norm{\Delta u_\lambda}_{L^\infty(0,T; H)}\leq c\,,
\eeq
\pier{whence, by comparison in \eqref{eq2_approx}, we also have}
\beq
  \label{est-lam4}
  \norm{\beta_\lambda(\partial_t u_\lambda)}_{L^\infty(0,T; H)}\leq c\,.
\eeq
Proceeding now as in \cite[\S~6]{bcst1}, we can conclude.

\subsection{The continuous dependence result}
We focus here on the proof of Theorem~\ref{contdep}.
Let $(u_i, \mu_i, \xi_i)$ satisfy \eqref{u}--\eqref{3}
with respect to the data $(u_{0,i}, g_i)$, for $i=1,2$: then, setting
  $u:=u_1-u_2$, $\mu:=\mu_1-\mu_2$, $\xi:=\xi_1-\xi_2$
  $u_0:=u_{0,1}-u_{0,2}$, and $g:=g_1-g_2$, we have
  \begin{align*}
  \partial_t u - \Delta\mu=0 \qquad&\text{in } Q\,,\\
  \mu=\eps\partial_t u + \xi - \delta\Delta u + \psi'(u_1) - \psi'(u_2) + g\qquad&\text{in } Q\,,\\
  u(0)=u_0 \qquad&\text{in } \Omega\,.
  \end{align*}
  Testing the first equation by $\mu$, the second by $\partial_t u $ and taking the difference we deduce,
  by monotonicity of $\beta$, for all $t\in[0,T]$,
  \begin{align*}
  &\int_{Q_t}|\nabla\mu|^2 + \eps\int_{Q_t}|\partial_t u|^2 + \int_{Q_t}\xi\partial_t u +
   \frac\delta2\int_\Omega|\nabla u(t)|^2\\
  &\leq \frac\delta2\norm{\nabla u_0}_H^2
  -\int_{Q_t}(\psi'(u_1)-\psi'(u_2))\partial_t u
  -\int_{Q_t}g\partial_t u\,.
  \end{align*}
  Now, the fact that $u_1, u_2 \in [a_0', b_0']\subset(a,b)$ for some $a_0', b_0'$ yields
  \[
  |\psi'(u_1)-\psi'(u_2)|\leq\norm{\psi''}_{C^0([a_0', b_0'])}|u|\,.
  \]
  Hence, using the Young inequality and the fact that 
  \[
  u(t)=u_0 + \int_0^t\partial_t u(s)\,ds\,,
  \]
  we are left with
  \begin{align*}
  &\int_{Q_t}|\nabla\mu|^2 + \frac\eps2\int_{Q_t}|\partial_t u|^2 
  + \int_{Q_t}\xi\partial_t u  +
  \frac\delta2\int_\Omega|\nabla u(t)|^2\\
  &\leq \frac\delta2\norm{\nabla u_0}_H^2+
  \frac1\eps\norm{\psi''}_{C^0([a_0', b_0'])}^2\int_{Q_t}|u|^2 +
  \frac1\eps\norm{g}_{L^2(0,T; H)}^2\\
  &\leq\frac\delta2\norm{\nabla u_0}_H^2 +
  \frac{2T}\eps\norm{u_0}_H^2 +
  \frac{2T}\eps \norm{\psi''}_{C^0([a_0', b_0'])}^2\int_0^t\int_{Q_s}|\partial_t u|^2\,ds
  +\frac1\eps\norm{g}_{L^2(0,T; H)}^2\,.
  \end{align*}
  The Gronwall lemma \pier{yields} then the desired continuous dependence estimate \eqref{dipcont}.


\section{Proof of Theorem~\ref{thm1}}\label{sec:limit-as-epssearrow0}
\setcounter{equation}{0}

This section is devoted to the proof of Theorem~\ref{thm1}.
Since $\delta>0$ is fixed and we let $\eps\searrow0$,
in order to avoid heavy notations we will not write
explicitly the dependence on $\delta$
for the quantities in play. In particular, let $(u_\eps, \mu_\eps, \xi_\eps)$
be any solution satisfying \eqref{u}--\eqref{3} for every $\eps>0$.

\subsection{First estimate}\label{ssec:est1_ep}
We test \eqref{1} by $\mu_\eps$, \eqref{2} by $\partial_t u_\eps$ and subtract, obtaining
\[\begin{split}
  \int_{Q_t}|\nabla\mu_\eps|^2 &+ \eps\int_{Q_t}|\partial_t u_\eps|^2 + \int_{Q_t}\xi_\eps\partial_t u_\eps
  +\frac\delta2\int_\Omega|\nabla u_\eps(t)|^2 + \int_\Omega\psi(u_\eps(t))\\
  &=\frac\delta2\int_\Omega|\nabla u_0|^2 + \int_\Omega\psi(u_0) -
   \int_{Q_t}g_\eps\partial_t u_\eps \qquad\forall\,t\in[0,T]\,.
\end{split}\]
Now, by \eqref{pier1}--\eqref{pier2} we have
\[
\psi(u_0)=\widehat\gamma(u_0) -\frac{K}{2}|u_0|^2 +\psi(r_0) +\frac{K}2|r_0|^2\,,
\]
where, recalling that $\psi'(u_0)\in H$ by \eqref{ip_eps1}, hence also 
$\gamma(u_0)\in H$,
\[
\widehat\gamma(u_0)\leq
\widehat\gamma(u_0) + \widehat{\gamma^{-1}}(\gamma(u_0))
=\gamma(u_0)u_0 \in L^1(\Omega)\,.
\]
Therefore, we see that
$\psi(u_0)\in L^1(\Omega)$.
By the monotonicity of $\beta$ and 
conditions \eqref{ip_eps1}, \eqref{ip_eps2} and \eqref{conv_g_eps},
integrating by parts in time the last term we infer that 
there exists $c>0$, independent of $\eps$, such that
\[
  \begin{split}
  &\int_{Q_t}|\nabla\mu_\eps|^2 + \eps\int_{Q_t}|\partial_t u_\eps|^2
  +\frac\delta2\int_\Omega|\nabla u_\eps(t)|^2 + C_1\int_\Omega|u_\eps(t)|^2\\
  &\leq c
  +\int_{Q_t}\partial_t g_\eps u_\eps - \int_\Omega g_\eps(t)u_\eps(t) + \int_\Omega g(0)u_0\\
  &\leq c
  +\int_{Q_t}|u_\eps|^2 + \frac14\int_Q|\partial_t g_\eps|^2
  +\frac{C_1}{2}\int_\Omega|u_\eps(t)|^2
  +\frac{1}{2C_1}\norm{g_\eps(t)}_{H}^2
  +\norm{g(0)}_{H}\norm{u_0}_H\,.
\end{split}
\]
Rearranging the terms \pier{and} recalling that $(g_\eps)_\eps$ is 
bounded in $H^1(0,T; H)$ independently of $\eps$ by \eqref{conv_g_eps},
\pier{an application of} the Gronwall lemma \pier{leads to}
\beq
  \label{est1_eps}
  \norm{\mu_\eps}_{L^2(0,T; V_0)} 
  + \norm{u_\eps}_{L^\infty(0,T; V)} 
  + \eps^{1/2}\norm{\partial_t u_\eps}_{L^2(0,T; H)}\leq c
\eeq
and by comparison in \eqref{1} \pier{we also deduce that}
\beq
  \label{est2_eps}
  \norm{\partial_t u_\eps}_{L^2(0,T; V_0^*)}\leq c\,.
\eeq

\subsection{Second estimate}\label{sec:second-estimate}
In order to derive this estimate 
first we need to identify the initial values of the solutions
$\mu_{0\eps}:=\mu_\eps(0)$ and $u_{0\eps}':=\partial_t u_\eps(0)$.

\begin{lem}
  For every $\eps>0$, there exists a unique triplet $(\mu_{0\eps}, u_{0\eps}', \xi_{0\eps})\in 
  W_0\times H\times H$ such that 
  \[
    u_{0\eps}' - \Delta\mu_{0\eps}=0\,, \quad
    \mu_{0\eps} = \eps u_{0\eps}' + \xi_{0\eps}
    -\delta\Delta u_0 + \psi'(u_0) + g(0)\,, \quad
    \xi_{0\eps} \in \beta(u_{0\eps}')
  \]
  \pier{almost everywhere in $\Omega$.}
  Moreover, there exists a positive constant $c$, independent of $\eps$, such that 
  \[
  \int_\Omega|\nabla \mu_{0\eps}|^2 + \eps\int_\Omega|u_{0\eps}'|^2 + 
  \int_\Omega\widehat{\beta^{-1}}(\xi_{0\eps}) \leq c \qquad\forall\,\eps>0\,.
  \]
\end{lem}
\begin{proof}
  Since $z_0:=-\delta\Delta u_0 + \psi'(u_0) + g(0)\in H$, existence and uniqueness of
  $(\mu_{0\eps}, u_{0\eps}', \xi_{0\eps})$ follows from the maximal monotonicity of $\beta$,
  arguing as in \cite[p.~1006]{bcst1}. 
  Moreover, testing the first equation by $\mu_{0\eps}$, the second by $u_{0\eps}'$
  and taking the difference we have
  \[
  \int_\Omega|\nabla\mu_{0\eps}|^2 + \eps\int_\Omega|u_{0\eps}'|^2
  +\int_\Omega\xi_{0\eps} u_{0\eps}' + \int_\Omega z_0 u_{0\eps}'=0\,.
  \]
  \luca{Since $\xi_{0\eps}\in\beta(u_{0\eps}')$,
  on the left-hand side we have that 
  $\xi_{0\eps}u_{0\eps}'= 
  \widehat\beta(u_{0\eps}') + \widehat{\beta^{-1}}(\xi_{0\eps})$.
  Moreover, by the Young inequality we have
  \begin{gather*}
  \int_\Omega|\nabla\mu_{0\eps}|^2 + \eps\int_\Omega|u_{0\eps}'|^2
  +\int_\Omega\widehat\beta(u_{0\eps}') + \int_\Omega\widehat{\beta^{-1}}(\xi_{0\eps})
  \\=
  -\int_\Omega z_0 u_{0\eps}'\leq \int_\Omega\widehat{\beta^{-1}}(-z_0) + \int_\Omega\widehat\beta(u_{0\eps}')\,,
   \end{gather*}
  from which the estimate follows thanks to hypothesis \eqref{ip_eps3}.}
\end{proof}

Now, we proceed formally, testing \eqref{1} by $\partial_t \mu_\eps$, 
the time-derivative of \eqref{2} by $\partial_t u_\eps$
and subtracting: a rigorous computation can be obtained through a discretization in time
(for further details, see for example \cite[\S~5.2]{bcst1}). 
We obtain then, recalling the previous lemma and that $\psi''\geq-K$ by \eqref{psi5},
\[\begin{split}
  &\frac12\int_\Omega|\nabla\mu_\eps(t)|^2 + \frac\eps2\int_\Omega|\partial_t u_\eps(t)|^2
  +\int_\Omega\widehat{\beta^{-1}}(\xi_\eps(t)) + \delta\int_{Q_t}|\nabla\partial_t u_\eps|^2\\
  &=\frac12\int_\Omega|\nabla\mu_{0\eps}|^2 + \frac\eps2\int_\Omega|u_{0\eps}'|^2 
  + \int_\Omega\widehat{\beta^{-1}}(\xi_{0\eps})
  -\int_{Q_t}\left(\partial_t g_\eps + \psi''(u_\eps)\partial_t u_\eps\right)\partial_t u_\eps\\
  &\leq c + \frac14\norm{\partial_t g_\eps}^2_{L^2(0,T; H)} 
  + \left(1+K\right)\int_{Q_t}|\partial_t u_\eps|^2\,.
\end{split}\]
By the compactness inequality \eqref{comp}, we can \pier{handle the last} term on the right-hand side~as
\[
  (1+K)\int_{Q_t}|\partial_t u_\eps|^2 \leq \frac\delta2\int_{Q_t}|\nabla\partial_t u_\eps|^2 + c\norm{\partial_t u_\eps}^2_{L^2(0,T; V_0^*)}\,,
\]
so that by \eqref{est2_eps} and again \eqref{comp} we infer (possibly renominating $c$) that 
\beq
  \label{est3_eps}
  \norm{\mu_\eps}_{L^\infty(0,T; V_0)} + \norm{\partial_t u_\eps}_{L^2(0,T; V)} + \eps^{1/2}\norm{\partial_t u_\eps}_{L^\infty(0,T; H)} \leq c
\eeq
and, by comparison in \eqref{1}, also
\beq
  \label{est4_eps}
  \norm{\partial_t u_\eps}_{L^\infty(0,T; V_0^*)} + \norm{\mu_\eps}_{L^2(0,T; H^3(\Omega))}\leq c\,.
\eeq

\subsection{Third estimate under assumption \eqref{ip_psi}}\label{sec:third-estimate-eps}
We test \eqref{2} by $-\delta\Delta\partial_t u_\eps + \partial_t\gamma(u_\eps)$:
to this end, note that since $\partial_t u_\eps\in L^2(0,T; V)$ only, then 
$-\Delta\partial_t u_\eps$ has to be interpreted as an element in $L^2(0,T; V^*)$.
However, be aware that $\xi_\eps\in L^\infty(0,T; H)$, so that the estimate that we perform is formal.
To be rigorous, one should regularize $\beta$ with its Yosida approximation $\beta_\lambda$ and then 
carry out the computations: 
as a matter of fact, it is readily seen that the resulting estimate 
would be independent of $\lambda$, so that we avoid such technicalities here. We have
\[\begin{split}
  &\frac12\int_\Omega|-\delta\Delta u_\eps + \gamma(u_\eps)|^2(t)\\
  &\qquad+\eps\delta\int_{Q_t}|\nabla\partial_t u_\eps|^2
  +\delta\int_{Q_t}\nabla \xi_\eps\cdot\nabla \partial_t u_\eps +
  \eps\int_{Q_t}\gamma'(u_\eps)|\partial_t u_\eps|^2+
  \int_{Q_t}\gamma'(u_\eps)\xi_\eps\partial_t u_\eps\\
  &=\frac12\int_\Omega|-\delta\Delta u_0 + \gamma(u_0)|^2
  +\delta\int_{Q_t}\nabla\mu_\eps\cdot\nabla\partial_t u_\eps
  +\int_{Q_t}\mu_\eps\gamma'(u_\eps)\partial_t u_\eps\\
  &\qquad+\int_{Q_t}(\partial_tg_\eps-K\partial_t u_\eps)(-\delta\Delta u_\eps + \gamma(u_\eps))
  -\int_\Omega(g_\eps(t)-Ku_\eps(t))(-\delta\Delta u_\eps+\gamma(u_\eps))(t)\\
  &\qquad+\int_\Omega(g(0)-Ku_0)(-\delta\Delta u_0+\gamma(u_0))
  \qquad\forall\,t\in[0,T]\,.
\end{split}\]
Now, as we have anticipated, if we replace $\beta$ with its Yosida approximation $\beta_\lambda$,
the third term on the left-hand side would give the contribution
\[
  \int_{Q_t}\beta_\lambda'(\partial_t u_{\eps,\lambda})|\nabla \partial_t u_{\eps,\lambda}|^2\geq 0 
  \qquad\forall\,\lambda>0\,.
\]
Moreover, it is also clear by the properties of $\beta$ that the last term on the left-hand side is nonnegative.
On the right hand side, the first term is finite by assumption \eqref{ip_eps1} while the second term
is bounded uniformly in $\eps$ by \eqref{est3_eps}. 
Furthermore, by \eqref{est1_eps}, \eqref{est3_eps}--\eqref{est4_eps} and \eqref{ip_psi},
using the continuous embeddings $V\embed L^6(\Omega)$ and $\pier{H^3(\Omega)}\embed L^\infty(\Omega)$ we have
\[\begin{split}
  \int_Q\mu_\eps\gamma'(u_\eps)\partial_t u_\eps&\leq
  \int_0^T\norm{\mu_\eps(t)}_{L^\infty(\Omega)}\norm{\gamma'(u_\eps(t))}_{L^{6/5}(\Omega)}\norm{\partial_t u_\eps(t)}_{L^6(\Omega)}\,dt\\
  &\leq c\norm{\mu_\eps}_{L^2(0,T;\pier{H^3(\Omega)})}\norm{\partial_t u_\eps}_{L^2(0,T; V)}\norm{\gamma'(u_\eps)}_{L^\infty(0,T; L^{6/5}(\Omega))}\\
  &\leq c(1+\norm{u_\eps}^5_{L^\infty(0,T; L^6(\Omega))}) \leq c
\end{split}\]
for a certain constant $c>0$ that we have updated step by step.
Finally, we handle the last three terms
on the right-hand side using Young's inequality, the estimate \eqref{est1_eps}
and the assumptions \eqref{ip_eps1} and \eqref{conv_g_eps} by
\[
  c + \int_{Q_t}|-\delta\Delta u_\eps + \gamma(u_\eps)|^2 + \frac14\int_\Omega|-\delta\Delta u_\eps + \gamma(u_\eps)|^2(t)\,.
\]
Consequently, rearranging the terms and using the Gronwall inequality \pier{lead to}
\[
  \norm{-\delta\Delta u_\eps + \gamma(u_\eps)}_{L^\infty(0,T; H)}\leq c\,.
\]
\luca{Since $\gamma$ is monotone,
testing $-\delta\Delta u_\eps + \gamma(u_\eps)$ by $-\Delta u_\eps$,
integrating by parts and using the Young inequality \pier{yield}
(recall that $\delta >0$ is fixed here)
\begin{align*}
  \delta\int_\Omega|\Delta u_\eps|^2&\leq 
  \delta\int_{\Omega}|\Delta u_\eps|^2 + \int_\Omega\gamma'(u_\eps)|\nabla u_\eps|^2
  =\int_\Omega(-\Delta u_\eps)(-\delta\Delta u_\eps + \gamma(u_\eps))\\
  &\leq
  \frac\delta2\int_\Omega|\Delta u_\eps|^2 + \frac1{2\delta}
  \int_\Omega|-\delta\Delta u_\eps + \gamma(u_\eps)|^2
\end{align*}
\pier{almost everywhere in $(0,T)$}.
Rearranging the terms and \pier{invoking} elliptic regularity we deduce~then}
\beq
  \label{est5_eps}
  \norm{u_\eps}_{L^\infty(0,T; W_{\bf n})} + \norm{\gamma(u_\eps)}_{L^\infty(0,T; H)} \leq c\,,
\eeq
\pier{and consequently}, by comparison in \eqref{2},
\beq
  \label{est6_eps}
  \norm{\xi_\eps}_{L^\infty(0,T; H)}\leq c\,.
  \eeq

\subsection{Third estimate under assumption \eqref{ip_beta}}
By \eqref{ip_beta} and \eqref{est3_eps} we immediately have
\beq
  \label{est7_eps}
  \norm{\xi_\eps}_{L^2(0,T; H)}\leq c\,\pier{.}
\eeq
Then, \pier{with the help of a comparison in \eqref{2} we see that}
\[
  \norm{-\delta\Delta u_\eps + \gamma(u_\eps)}_{L^2(0,T; H)}\leq c\,\pier{.}
\]
\pier{Hence, by applying the same argument leading to \eqref{est5_eps} we arrive at} 
\beq
  \label{est8_eps}
  \norm{u_\eps}_{L^2(0,T; W_{\bf n})} + \norm{\gamma(u_\eps)}_{L^2(0,T; H)}\leq c\,.
\eeq

\subsection{Passage to the limit}
\label{ssec:limit_eps}
Let us assume first \eqref{ip_psi}. Then, by the estimates \eqref{est1_eps}--\eqref{est6_eps}
we deduce that there is a triplet $(u,\mu,\xi)$ such that
\begin{gather*}
  u \in W^{1,\infty}(0,T; V_0^*)\cap H^1(0,T; V)\cap L^\infty(0,T; W_{\bf n})\,,\\
  \mu \in L^\infty(0,T; V_0)\cap L^2(0,T; H^3(\Omega))\,, \qquad
  \xi \in L^\infty(0,T; H)\,, \qquad \eta \in L^\infty(0,T; H)
\end{gather*}
and, along a subsequence that we still denote by $\eps$ for simplicity,
\begin{gather*}
  u_\eps \wstarto u \quad\text{in } W^{1,\infty}(0,T; V_0^*)\cap L^\infty(0,T; W_{\bf n})\,,\qquad
  u_\eps\wto u \quad\text{in } H^1(0,T; V)\,,\\
  \mu_\eps \wstarto\mu \quad\text{in } L^\infty(0,T; V_0)\,,\qquad
  \mu_\eps \wto \mu \quad\text{in } L^2(0,T; H^3(\Omega))\,,\\
  \xi_\eps \wstarto \xi \quad\text{in } L^\infty(0,T; H)\,, \qquad 
  \gamma(u_\eps)\wstarto \eta \quad\text{in } L^\infty(0,T; H)\,,\\
  \eps \partial_t u_\eps \to 0 \quad\text{in } L^\infty(0,T; H)\,.
\end{gather*}
Now, from the first two convergences and a classical compactness result
(see e.g.~\cite[Cor.~4, p.~85]{simon}), we have that 
\[
  u_\eps \to u \quad\text{in } C^0([0,T];V)\,,
\]
which immediately implies that $\eta=\gamma(u)=\psi'(u)+Ku$ a.e.~in $Q$
by the strong-weak closure \pier{(see, e.g., \cite[Prop.~2.1, p.~29]{barbu-monot})}
of the maximal monotone operator $\gamma$.

Moreover, letting then $\eps\searrow0$
in \eqref{1}--\eqref{2}, we infer by the weak convergences that
\eqref{1lim}--\eqref{2lim} hold.
Now, proceeding as in the first estimate we have that
\begin{align*}
  \int_Q\xi_\eps\partial_t u_\eps &\leq \frac\delta2\int_\Omega|\nabla u_0|^2 + \int_\Omega\widehat\gamma(u_0)
  -\int_Q g_\eps\partial_t u_\eps + K\int_{Q}u_\eps\partial_t u_\eps\\
  &\quad- \int_Q|\nabla\mu_\eps|^2 - \frac\delta2\int_\Omega|\nabla u_\eps(T)|^2 - \int_\Omega\widehat\gamma(u_\eps(T))
\end{align*}
so that, using the convergences already proved and the (weak) lower
semicontinuity of the convex integrands, we infer that
\[\begin{split}
  \limsup_{\eps\searrow0}\int_Q\xi_\eps\partial_t u_\eps&\leq
  \frac\delta2\int_\Omega|\nabla u_0|^2 + \int_\Omega\widehat\gamma(u_0)
  -\int_Q g\partial_tu + K\int_Qu\partial_t u\\
  &\quad- \int_Q|\nabla\mu|^2 - \frac\delta2\int_\Omega|\nabla u(T)|^2 - \int_\Omega\widehat\gamma(u(T))\,.
\end{split}\]
Since we have already proved \eqref{1lim}--\eqref{2lim}, performing the 
same computation on the limit problem yields that the right-hand side
is exactly $\int_Q\xi\partial_t u$. Hence,
\[
  \limsup_{\eps\searrow0}\int_Q\xi_\eps\partial_t u_\eps\leq\int_Q\xi\partial_t u\,,
\]
and this is enough to conclude that $\xi\in\beta(\partial_t u)$ a.e.~in $Q$.

If \eqref{ip_beta} is in order, we can proceed in exactly the same way using the estimates
\eqref{est7_eps}--\eqref{est8_eps} instead of \eqref{est5_eps}--\eqref{est6_eps}:
note that in this case, by \cite[Cor.~4, p.~85]{simon} 
we can only infer the strong convergence
\[
  u_\eps \to u \quad\text{in } C^0([0,T]; H)\pier{{}\cap L^2(0,T;V)}\,.
\]
Since also $u_\eps(T)\wto u(T)$ in $V$, the argument performed above still works
by weak lower semicontinuity.


\section{Proof of Theorem~\ref{thm2}}\label{sec:delta-to-0}
This section is devoted to the proof of Theorem~\ref{thm2}. We shall consider $\eps>0$ fixed and we will not write explicitly the dependence on $\eps$ for the quantities in play. Thus, in what follows we shall let  $(u_\delta,\mu_\delta,\xi_\delta)$ be the solution to \eqref{u}--\eqref{3} with respect to the data $(u_{0\delta}, g_\delta)$ for every $\delta >0$.

\subsection{First estimate}
To obtain the first estimate, proceed as in Section~\ref{sec:limit-as-epssearrow0}: 
we test \eqref{1} by $\mud$ and we subtract \eqref{2} tested by $\partial_t\ud$. 
By integration over $(0,t)$ for $t\in[0,T]$, we obtain
\[\begin{split}
  \int_{Q_t}|\nabla\mud|^2 &+ \eps\int_{Q_t}|\partial_t \ud|^2 + \int_{Q_t}\xid\partial_t \ud
  +\frac\delta2\int_\Omega|\nabla \ud(t)|^2 + \int_\Omega\psi(u_\delta(t))\\
&=\frac\delta2\int_\Omega|\nabla u_{0\delta}|^2 
+ \int_\Omega\psi(u_{0\delta})-\int_{Q_t}g_\delta\partial_t \ud\,.
\end{split}
\]
Since $\psi'(u_{0\delta})$ is bounded in $H$ by \eqref{ip_delta4}, hence also
$\psi(u_{0\delta})$ is bounded in $L^1(\Omega)$ as already 
pointed out at the beginning of Section~\ref{ssec:est1_ep}\pier{. Then,}
the first two terms on the right-hand side are bounded uniformly in $\delta$.
Moreover, one has
\[
\int_{Q_t}g_\delta\partial_t\ud\le \frac1{2\eps}\norm{g_\delta}^2_{L^2(0,T; H)}+\frac\eps2\int_{Q_t}|\partial_t\ud|^2
\]
Consequently, recalling also that 
\[
  \widehat\beta(\partial_t\ud)\leq \widehat\beta(\partial_t\ud) 
  + \widehat{\beta^{-1}}(\xi_\delta)=\xi_\delta\partial_t\ud\,,
\]
by the assumption \eqref{ip_delta5} on $(g_\delta)_\delta$ and the Gronwall lemma we deduce that
\begin{equation}\label{est1_delta}
  \norm{\mud}_{L^2(0,T; V_0)} +\norm{\partial_t\ud}_{L^2(0,T; H)}+\delta^{1/2}\norm{\ud}_{L^\infty(0,T;V)}+
  \norm{\widehat\beta(\partial_t\ud)}_{L^1(Q)}\leq c,
\end{equation}
where $c$ is a positive constant independent of $\delta$.

\subsection{Second estimate}
We repeat the same estimate as in Section~\ref{sec:second-estimate}.
First of all, we need to identify and estimate the initial values of the solutions
$\mu_{0\delta}:=\mu_\delta(0)$ and $u_{0\delta}':=\partial_t u_\delta(0)$.

\begin{lem}
  For every $\delta>0$, there exists a unique triplet 
  $(\mu_{0\delta}, u_{0\delta}', \xi_{0\delta})\in W_0\times H\times H$ such that 
  \[
    u_{0\delta}' - \Delta\mu_{0\delta}=0\,, \quad
    \mu_{0\delta} = \eps u_{0\delta}' + \xi_{0\delta}
    -\delta\Delta u_{0\delta} + \psi'(u_{0\delta}) + g_\delta(0)\,, \quad
    \xi_{0\delta} \in \beta(u_{0\delta}')
  \]
  \pier{almost everywhere in $\Omega.$}
  Moreover, there exists a positive constant $c$, independent of $\delta$, such that 
  \[
  \int_\Omega|\nabla \mu_{0\delta}|^2 + \eps\int_\Omega|u_{0\delta}'|^2 + 
  \int_\Omega\widehat{\beta^{-1}}(\xi_{0\delta}) \leq c \qquad\forall\,\delta>0\,.
  \]
\end{lem}
\begin{proof}
  Since $z_{0\delta}:=-\delta \Delta u_{0\delta} + \psi'(u_{0\delta}) + g_\delta(0)\in H$, 
  the existence and uniqueness of
  $(\mu_{0\delta}, u_{0\delta}', \xi_{0\delta})$ follows from the maximal monotonicity of $\beta$,
  arguing as in Section~\ref{sec:second-estimate}. 
  Moreover, testing the first equation by $\mu_{0\delta}$, the second by $u_{0\delta}'$
  and taking the difference we have
  \[
  \int_\Omega|\nabla\mu_{0\delta}|^2 + \eps\int_\Omega|u_{0\delta}'|^2
  +\int_\Omega\xi_{0\delta} u_{0\delta}' + \int_\Omega z_{0\delta} u_{0\delta}'=0\,.
  \]
  By monotonicity of $\beta$, the fact that $z_{0\delta}$ is bounded in $H$ 
  thanks to the assumptions \eqref{ip_delta4}--\eqref{ip_delta5},
  and the Young inequality we have
  \[
  \int_\Omega|\nabla\mu_{0\delta}|^2 + \frac\eps2\int_\Omega|u_{0\delta}'|^2
  +\int_\Omega\widehat\beta(u_{0\delta}') + \int_\Omega\widehat{\beta^{-1}}(\xi_{0\delta})=
  -\int_\Omega z_0 u_{0\delta}'\leq \pier{\frac1{2\eps}}\norm{z_{0\delta}}_H^2\leq c\,,
  \]
  from which the estimate follows.
\end{proof}

Performing then the same \pier{computations} as in Section~\ref{sec:second-estimate} we \pier{deduce that}
\[\begin{split}
  &\frac12\int_\Omega|\nabla\mud(t)|^2 + \frac\eps2\int_\Omega|\partial_t \ud(t)|^2
  +\int_\Omega\widehat{\beta^{-1}}(\xid (t)) 
+ \delta\int_{Q_t}|\nabla\partial_t \ud|^2
+\int_{Q_t}\gamma'(\ud)|\partial_t\ud|^2\\
  &=\frac12\int_\Omega|\nabla\mu_{0\delta}|^2 + \frac\eps2\int_\Omega|u_{0\delta}'|^2 + \int_\Omega\widehat{\beta^{-1}}(\xi_{0\delta})
  -\int_{Q_t}\left(\partial_t g_\delta - K\partial_t \ud\right)\partial_t \ud\\
  &\leq c + \frac1{2}\norm{\partial_t g_\delta}^2_{L^2(0,T; H)} + \left(\frac 1 2+K\right)\int_0^t\int_{\Omega}|\partial_t \ud|^2\,.
\end{split}\]
As a result, we obtain the following estimate
\begin{multline}\label{est2_delta}
    \norm{\mud}_{L^\infty(0,T;V_0)} +  \norm{\ud}_{W^{1,\infty}(0,T; H)}+\|\widehat{\beta^{-1}}(\xid)\|_{L^\infty(0,T;L^1(\Omega))}+\sqrt\delta\norm{\ud}_{H^1(0,T;V)}\leq c,
\end{multline}
whence, by comparison in \eqref{1}, the inequality
\begin{equation}\label{est3_delta}
\|\mud\|_{L^\infty(0,T; W_0)}\le c.
\end{equation}

\subsection{Third estimate}
The purpose of this subsection is to show that if the initial data satisfies
the boundedness assumption \eqref{ip_delta3}
then $u_\delta$ stays bounded in an interval $[a_0',b_0']\subset(a,b)$
uniformly in $\delta$, namely
\[
  a_0'\leq u_\delta\le b_0',\quad\text{a.e.~in } Q,
\]
with $a_0', b_0'\in\erre$ independent of $\delta$ and $[a_0,b_0]\subseteq[a_0', b_0']\subset(a,b)$. 
The idea here is to apply the maximum principle to a 
nonlinear elliptic system that arises from a time-discretization of \eqref{eq2:intro},
as in Section~\ref{sec:prelim}.

To this end, let us note that, thanks to the estimate \eqref{est3_delta},
the continuous embedding $\pier{W_0}\embed L^\infty(\Omega)$ and 
assumption \eqref{ip_delta5}, there exists a positive constant $M$,
independent of $\delta$ such that
\begin{equation}\label{est_inf_delta}
  \|\mud-g_\delta\|_{L^\infty(Q)}\leq M.
\end{equation}
Now, in principle the constants $a_0'$ and $b_0'$ given by Theorem~\ref{thm-prelim}
may depend on $\delta$.
However, 
going back to Section~\ref{sec:prelim}, we note that 
the choice of the constants $\bar a, \bar b, a_0', b_0'$ only 
depends on $\norm{\mu_{\delta\lambda}-g_{\delta}}_{L^\infty(Q)}$
and the behaviour of $\psi$.
Hence, the uniform estimate \eqref{est_inf_delta}
implies that $a_0'$ and $b_0'$ can be chosen independently of 
the parameter $\delta$.

As a consequence, we deduce that there \pier{exist} $a_0', b_0' \in\erre$,
independent of $\delta$, with $[a_0,b_0]\subseteq[a_0', b_0']\subset(a,b)$, such that  
\beq\label{est4_delta}
  a_0'\leq u_\delta\leq b_0'\quad\text{a.e.~in } Q\,, \quad\forall\,\delta>0\,\pier{.}
\eeq
\pier{Hence, since $\psi'\in C^1([a_0', b_0'])$, we also have}
\beq\label{est5_delta}
  \norm{\psi'(u_\delta)}_{L^\infty(Q)}+ \norm{\psi''(u_\delta)}_{L^\infty(Q)}\leq c\,.
\eeq

Arguing now as in Section~\ref{sec:prelim} in order to prove \eqref{est-lam3}--\eqref{est-lam4}, 
i.e.~formally testing \eqref{2} by $-\delta^{1/2}\Delta\partial_t u_\delta$, we have by the Young inequality 
and estimate \eqref{est5_delta} that 
\begin{align*}
  \eps\delta^{1/2}&\int_{Q_t}|\nabla\partial_t u_\delta|^2 
  +\delta^{1/2}\int_{Q_t}\nabla\xi_\delta\cdot\nabla\partial_t u_\delta
  + \frac{\delta^{3/2}}2\int_{\Omega}|\Delta u_\delta(t)|^2\\
  &\leq
  \frac{\delta^{3/2}}2\int_{\Omega}|\Delta u_{0\delta}|^2 
  -\delta^{1/2} \int_0^t{}_{V^*}\ip{\Delta\partial_t\ud}{\mud-\psi'(\ud)-g_\delta}_V\\
  &=\frac{\delta^{3/2}}2\int_{\Omega}|\Delta u_{0\delta}|^2 
  + \int_{Q_t}(\nabla\mud-\psi''(\ud)\nabla\ud-\nabla g_\delta)\cdot(\delta^{1/2}\nabla\partial_t \ud) \\
  &\leq c\left(\delta^{3/2}\int_{\Omega}|\Delta u_{0\delta}|^2
  +\norm{\mu_\delta}_{L^2(0,T; V_0)}^2
  +\delta^{1/2}\norm{g_\delta}_{L^2(0,T;V)}^2
  \right)\\
  &\quad+ \frac{\eps\delta^{1/2}}2\int_{Q_t}|\nabla\partial_t\ud|^2
  +c\delta^{1/2}\int_{Q_t}|\nabla \ud|^2\,.
\end{align*}
Hence, rearranging the terms, using the assumptions \eqref{ip_delta1}--\eqref{ip_delta5},
together with the monotonicity of $\beta$,
we infer that 
\[
  \frac{\eps}2\delta^{1/2}\int_{Q_t}|\nabla\partial_t u_\delta|^2 
  + \frac{\delta^{3/2}}2\int_{\Omega}|\Delta u_\delta(t)|^2
  \leq c\left(1 + \delta^{1/2}\int_{Q_t}|\nabla \ud|^2\right)\,.
\]
Writing $\ud=u_{0\delta}+\int_0^\cdot\partial_t\ud(s)\,ds$ and recalling the assumption
\eqref{ip_delta4}, we deduce that (updating the constant $c$ at each step)
\begin{align*}
  &\frac{\eps\delta^{1/2}}2\int_{Q_t}|\nabla\partial_t u_\delta|^2 
  + \frac{\delta^{3/2}}2\int_{\Omega}|\Delta u_\delta(t)|^2\\
  &\leq c\left(1 + 2\delta^{1/2}\norm{\nabla u_{0\delta}}_H^2+
  2T\delta^{1/2}\int_0^t\int_{Q_s}|\nabla\partial_t \ud|^2\,ds\right)\\
  &\leq c\left(1+ \delta^{1/2}\int_0^t\int_{Q_s}|\nabla\partial_t \ud|^2\,ds\right)\,.
\end{align*}
Hence, by the Gronwall lemma we deduce also the estimate
\beq
  \label{est6_delta}
  \delta^{1/4}\norm{\nabla\partial_t\ud}_{L^2(0,T; H)}+\delta^{3/4}\norm{\Delta\ud}_{L^\infty(0,T; H)}\leq c\,,
\eeq
and, by comparison in \eqref{2},
\beq
  \label{est7_delta}
  \norm{\xi_\delta}_{L^\infty(0,T; H)}\leq c\,.
\eeq

\subsection{Passage to the limit}
From the a priori estimates \eqref{est1_delta}--\eqref{est7_delta}, 
using standard compactness results we have the following convergences, up to a subsequence,
\[
  \begin{aligned}
u_\delta\stackrel{*}\rightharpoonup u \quad&\text{in }W^{1,\infty}(0,T;H)\cap L^\infty(Q)\,,\\
\mu_\delta\stackrel{*}\rightharpoonup \mu \quad&\text{in }L^\infty(0,T;W_0)\,,\\    
\delta^{1/2}u_\delta\to 0 \quad&\text{in }H^1(0,T; V)\,,\\
\delta\ud \to 0 \quad&\text{in }L^\infty(0,T; W_{\bf n})\,,\\
\xi_\delta\stackrel{*}\rightharpoonup \xi \quad&\text{in }L^\infty(0,T; H).\\
  \end{aligned}
\]

In order to pass to the limit 
in the equation \eqref{2}, we need to prove now 
a strong convergence for the sequence $(u_\delta)_\delta$.
We take the difference of \eqref{1}--\eqref{2} for 
two different indexes $\delta$ and $\delta'$: then,
we test the first equation by $\mu_\delta-\mu_{\delta'}$
and the second by $\partial_t(u_\delta-u_{\delta'})$, obtaining 
\begin{align*}
  &\int_{Q_t}|\nabla(\mu_\delta-\mu_{\delta'})|^2
  +\eps\int_{Q_t}|\partial_t(u_\delta-u_{\delta'})|^2
  +\int_{Q_t}(\xi_\delta-\xi_{\delta'})(\partial_t u_\delta - \partial_t u_{\delta'})\\
  &\leq S_{\delta,\delta'}(t)
  -\int_{Q_t}(\psi'(u_\delta) - \psi'(u_{\delta'}))\partial_t(u_\delta - u_{\delta'})
\end{align*}
where
\[
 S_{\delta,\delta'}(t):=\int_{Q_t}(\delta\Delta u_\delta - \delta'\Delta u_{\delta'})\partial_t(u_\delta - u_{\delta'})
 -\int_{Q_t}(g_\delta - g_{\delta'})\partial_t(u_\delta - u_{\delta'})\,.
\]
Note that by \eqref{est6_delta} and \eqref{ip_delta6} we have that 
$(S_{\delta,\delta'})_{\delta,\delta'}$ is uniformly bounded in $W^{1,\infty}(0,T)$
and converges pointwise to $0$ due to the convergences above.
Hence, we deduce that 
\[
  S_{\delta,\delta'} \to 0 \quad\text{in } C^0([0,T])\,.
\]
Moreover, since $\psi'\in C^1([a_0',b_0'])$, we have
\begin{align*}
   &\int_{Q_t}|\psi'(u_\delta)-\psi'(u)||\partial_t(u_\delta-u_{\delta'})|\\
   &\leq
   \frac \eps2 \int_Q |\partial_t(u_\delta-u_{\delta'})|^2 + 
   \frac1{2\eps}\norm{\psi''}^2_{L^\infty(a_0',b_0')}\int_{Q_t}|u_\delta-u_{\delta'}|^2\,,
\end{align*}
where 
\[
  \int_{Q_t}|u_\delta-u_{\delta'}|^2
  \leq 2T\norm{u_{0\delta}-u_{0\delta'}}_H^2 +2T \int_0^t\int_{Q_s}|\partial_t(u_\delta-u_{\delta'})|^2\,ds\,.
\]
Hence, rearranging the terms and using the monotonicity of $\beta$ and \eqref{incl}, we have \pier{that}
\begin{align*}
  &\int_{Q_t}|\nabla(\mu_\delta-\mu_{\delta'})|^2
  +\frac{\eps}2\int_{Q_t}|\partial_t(u_\delta-u_{\delta'})|^2\\
  &\leq \norm{S_{\delta,\delta'}}_{L^\infty(0,T)} + 
  c\left(\norm{u_{0\delta}-u_0}_H^2 + \int_0^t\int_{Q_s}|\partial_t(u_\delta-u)|^2\,ds\right)
\end{align*}
for a positive constant $c$, independent of $\delta$ and $\delta'$.
The Gronwall lemma yields then
\beq\label{error}
  \norm{\mu_\delta-\mu_{\delta'}}_{L^2(0,T; V_0)}^2 + 
  \norm{\partial_t(u_\delta-u_{\delta'})}^2_{L^2(0,T; H)}\leq
  c\left(\norm{S_{\delta,\delta'}}_{L^\infty(0,T)} + \norm{u_{0\delta}-u_{0\delta'}}_H^2\right)
\eeq
possibly updating the value of $c$. Recalling again \eqref{ip_delta6}, we deduce that 
\[
  u_\delta\to u\quad \text{in }H^1(0,T;H)\,,\qquad
  \mu_\delta\to \mu\quad \text{in }L^2(0,T;V_0)\,.
\]
In particular, we have the convergence
\[
  \int_Q |\psi'(u_\delta)-\psi'(u)|^2\leq \norm{\psi''}_{L^\infty(a_0',b_0')}^2\int_Q |u_\delta-u|^2\to 0.
\]
Therefore, the strong convergence of $u_\delta$ and 
the weak convergence of $\xi_\delta$ to $\xi$ 
allow us to prove \pier{\eqref{pier6}, i.e. the inclusion $\xi\in\beta(\partial_t u)$}, by maximal monotonicity. \pier{Then, passing to the limit 
in \eqref{1}--\eqref{3} as $\delta\searrow0$,} we can conclude. 

Finally, note that letting $\delta'\searrow0$ in \eqref{error} 
\pier{and taking \eqref{est6_delta} into account, by the Young inequality
we obtain}
\[
  \norm{\mu_\delta-\mu}_{L^2(0,T; V_0)}^2 + 
  \norm{\partial_t(u_\delta-u)}^2_{L^2(0,T; H)}\leq
  c\left(\norm{u_{0\delta}-u_0}_H^2
  +\norm{g_\delta-g}^2_{L^2(0,T; H)}
  +\delta^{1/2}\right)\,,
\]
\pier{that is nothing but \eqref{errorest}. Thus, we conclude the proof of Theorem~\ref{thm2}.}

\section*{Acknowledgments}
\pier{EB and PC gratefully acknowledge some 
financial support from the GNAMPA (Gruppo Nazionale per l'Analisi Matematica, 
la Probabilit\`a e le loro Applicazioni) of INdAM (Isti\-tuto 
Nazionale di Alta Matematica) and the IMATI -- C.N.R. Pavia.
Moreover, PC recognizes the contribution by the Italian Ministry of Education, 
University and Research~(MIUR): Dipartimenti di Eccellenza Program (2018--2022) 
-- Dept.~of Mathematics ``F.~Casorati'', University of Pavia.} 
LS has been funded by the Vienna Science and 
Technology Fund (WWTF) through Project MA14-009. 
GT acknowledges the support of INdAM's GNFM (Gruppo Nazionale per la Fisica Matematica) and the Grant of Excellence Departments, MIUR-Italy (Art.$1$, commi $314$-$337$, Legge $232$/$2016$).



\begin{thebibliography}{10}

\bibitem{Agosti2017}
A.~\pier{Agosti,} P.F.~Antonietti, P.~Ciarletta, M.~Grasselli, and M.~Verani. 
A Cahn-Hilliard-type equation with application to tumor growth dynamics.
{\it Math. \pier{Methods} Appl. Sci.} {\bf 40} (2017), 7598--7626.

\bibitem{Bagagiolo2000}
F.~Bagagiolo and A.~Visintin. 
\pier{Hysteresis in filtration through porous media.
{\em Z. Anal. Anwendungen}} {\bf 19} (2000), 977--997.

\bibitem{barbu-monot}
V.~Barbu. 
\pier{``}Nonlinear differential equations of monotone types in Banach spaces\pier{''}. 
{\em Springer Monographs in Mathematics. Springer, New York} (2010).

\bibitem{BCT}
E.~Bonetti, P.~Colli and G.~Tomassetti.
A non-smooth 
regularization of a forward-backward parabolic equation. 
{\it Math. Models Methods Appl. Sci.} {\bf 27} (2017), 641--661.

\bibitem{bcst1}
E.~Bonetti, P.~Colli, L.~Scarpa and G.~Tomassetti.
A doubly nonlinear Cahn-Hilliard system with nonlinear viscosity.
{\em Commun. Pure Appl. Anal.} {\bf 17} (2018), 1001--1022.

\bibitem{Bertozzi2007}
A.L. Bertozzi, S.Esedoglu, and A.~Gillette. 
Inpainting of binary images using the Cahn--Hilliard equation. 
\pier{{\it IEEE Trans. Image Process.}}
{\bf 16} (2007), 285--291.

\bibitem{Botkin2016}
N.D.~Botkin, \pier{M.~Brokate and} E.~El Behi-Gornostaeva. 
One-phase flow in porous media with hysteresis. 
{\it Physica B} {\bf 486} (2016), 183--186.

\bibitem{Cahn1961}
J.W.~Cahn. 
On spinodal decomposition. 
{\it Acta Metall.}  {\bf 9} (1961), 795--801.

\bibitem{Cahn1958}
J.W.~Cahn and J.E.~Hilliard. 
Free energy of a nonuniform system. {I}. {I}nterfacial free energy. 
{\it J. Chem. Phys.} {\bf 28} (1958), 258--267.

\pier{%
\bibitem{CGPS1}	
P. Colli, G. Gilardi, P. Podio-Guidugli and J. Sprekels. 
Well-posedness and long-time behavior for a nonstandard viscous Cahn--Hilliard system.
{\it SIAM J. Appl. Math.} {\bf 71} (2011), 1849--1870. 	
%
\bibitem{CGPS2}		
P. Colli, G. Gilardi, P. Podio-Guidugli and J. Sprekels. 
Global existence and uniqueness for a singular/degenerate Cahn--Hilliard system 
with viscosity. {\it J. Differential Equations} {\bf 254} (2013), 4217--4244.	
%
\bibitem{CGRS1}
P. Colli, G. Gilardi, E. Rocca and  J. Sprekels.
Vanishing viscosities and error estimate for a Cahn--Hilliard type phase field system related to tumor growth.
{\it Nonlinear Anal. Real World Appl.} {\bf 26} (2015), 93--108.	
%
\bibitem{CGRS2}
P. Colli, G. Gilardi, E. Rocca and  J. Sprekels.
Asymptotic analyses and error estimates for a Cah--Hilliard type phase field system modelling tumor growth.
{\it Discrete Contin. Dyn. Syst. Ser. S} {\bf 10} (2017), 37--54.%
}

\bibitem{colli-scar16}
P.~Colli and L.~Scarpa. 
From the viscous Cahn--Hilliard equation to a regularized forward-backward parabolic equation. 
{\em Asymptot. Anal.} {\bf 99} (2016), 183--205. 

\bibitem{Fife2001} 
P.C.~Fife. 
Models \pier{for} phase separation and their mathematics. 
\pier{{\it Electron. J. Differential Equations} {\bf 48} (2000), 26 pp.}

\pier{\bibitem{GGM}
C.G.~Gal, M. Grasselli and A. Miranville.  
Cahn-Hilliard-Navier-Stokes systems with moving contact lines.
{\em Calc. Var. Partial Differential Equations} {\bf 55} (2016), Art. 50, 47 pp.}

\bibitem{Gurtin1996}
M.E.~Gurtin. 
Generalized Ginzburg-Landau and Cahn-Hilliard equations based on a microforce balance. 
\pier{{\em Phys. D\/}} {\bf 92} (1996), 178--192.

\bibitem{Latroce} 
M.~Latroche. 
Structural and thermodynamic properties of metallic hydrides used for energy storage. 
{\it J. Phys. Chem. Solids} {\bf 65} (2004) 517--522.

\bibitem{lions}
J.~L.~Lions.
\pier{``}Quelques m\'ethodes de r\'esolution des probl\`emes aux limites non lin\'eaires\pier{''}.
{\em Dunod; Gauthier-Villars, Paris} (1969).

\bibitem{Liu2016}
Q.X.~Liu, M.~Rietkerk, P.M.J.~Herman, T.~Piersma, J.M.~Fryxell, and J.~van de Koppel. 
Phase separation driven by density-dependent movement: A novel mechanism for ecological patterns. 
{\it Phys. Life Rev.}, {\bf 19} (2016) 107--121. 

\bibitem{Miranville2000}
A.~Miranville. 
Some generalizations of the Cahn--Hilliard equation. 
{\it \pier{Asymptot.} Anal.} {\bf 22} (2000) 235--259.

\bibitem{Miranville2010}
A.~Miranville and G.~Schimperna. 
On a doubly nonlinear Cahn-Hilliard-Gurtin system. 
\pier{{\it Discrete Contin. Dyn. Syst. Ser. B}} {\bf 14} (2010), 675--697.

\bibitem{Miranville2009}
A.~Miranville and S. Zelik. 
Doubly nonlinear Cahn-Hilliard-Gurtin equations. 
\pier{\emph{Hokkaido Math. J.}} {\bf 38} (2009), 315--360.

\bibitem{Miranville2017}
A.~Miranville. 
The Cahn--Hilliard equation and some of its variants. 
{\it AIMS Math.} \pier{{\bf 2}} (2017), 479--544. 

\bibitem{Novic1988viscous}
A.~Novick-Cohen. 
On the viscous {C}ahn-{H}illiard equation, in
``Material instabilities in continuum mechanics ({E}dinburgh, 1985--1986)''.
{\it Oxford Sci. Publ.}, Oxford Univ. Press, New York, (1988), 329--342.

\bibitem{NovicP1991TAMS}
A.~Novick-Cohen and R.~L. Pego. 
Stable patterns in a viscous diffusion equation. 
{\it Trans. Amer. Math. Soc.} {\bf 324} (1991), 331--351.  

\pier{\bibitem{Podio}
P.~Podio-Guidugli. 
Models of phase segregation and diffusion of atomic 
species on a lattice. {\it Ric. Mat.} {\bf 55} (2006), 105--118.}
  
\bibitem{roub}
T.~Roub\'{\i}\v{c}ek.
\pier{``}Nonlinear partial differential equations with applications\pier{''}.
{\em Birkh\"{a}user Verlag, Basel} (2005).

\bibitem{scar19}
L.~Scarpa.
Existence and uniqueness of solutions to singular Cahn--Hilliard 
equations with nonlinear viscosity terms and dynamic boundary conditions.
{\em  J. Math. Anal. Appl.} {\bf 469} (2019), 730--764.

\bibitem{Schweizer2017}
B.~Schweizer. 
Hysteresis in porous media: \pier{modelling} and analysis. 
\pier{{\em Interfaces Free Bound.}} {\bf 19} (2017), 417--447.

\bibitem{simon}
J.~Simon.
Compact sets in the space {$L^p(0,T;B)$}.
{\em Ann. Mat. Pura Appl. (4)} {\bf 146} (1987), 65--96.

\bibitem{Tomas}
G.~Tomassetti. 
Smooth and non-smooth regularizations
of the nonlinear diffusion equation. 
{\em Discrete Contin. Dyn. Syst. Ser. S\/} {\bf 10} (2017) 1519-1537.

\end{thebibliography}
\end{document}